%% file: convex_arxiv_20260317.tex
\documentclass[12pt]{amsart}
\usepackage{graphicx}
\usepackage{amsthm}
\usepackage{amssymb}
\usepackage{verbatim}
\usepackage{enumerate}
\title{Convex sets and axiom of choice}
\date{}
\author{Yasuo~Yoshinobu}
\thanks{The author is partially supported by JSPS KAKENHI Grant Number 18K03394.}
\address{Graduate School of Informatics, Nagoya University, Furo-cho, Chikusa-ku, Nagoya 464-8601, Japan}
\email{yosinobu@i.nagoya-u.ac.jp}
\theoremstyle{definition}
\newtheorem{dfn}{Definition}[section]
\newtheorem{prop}[dfn]{Proposition}
\newtheorem{thm}[dfn]{Theorem}
\newtheorem{lma}[dfn]{Lemma}
\newtheorem{cor}[dfn]{Corollary}

\newtheorem{qtn}[dfn]{Question}

\newcommand{\restrict}{\upharpoonright}

\newcommand{\seq}[1]{{\langle#1\rangle}}

\newcommand{\R}{{\mathbb R}}
\newcommand{\Q}{{\mathbb Q}}

\makeatletter
\renewcommand{\p@enumii}{}
\makeatother
\begin{document}
\begin{abstract}
\input{convex_abstract.tex}
\end{abstract}
\maketitle
\section{Introduction}
\input{convex_intro.tex}
\input{convex_notations.tex}

\input{convex_faces.tex}

\input{convex_choices.tex}

\section{Basic facts about $\mathrm{MCV}$}
\input{convex_basics.tex}

\input{convex_generalchoice.tex}

\input{convex_fullmcv.tex}

\section{$\mathrm{MCV}(2)$ and $\mathrm{MCV}(3)$}
\input{convex_2and3.tex}

\section{Faces and convex filtrations}
\input{convex_generalapproach.tex}

\section{$\mathrm{CC}_\R$ implies $\mathrm{MCV}(2)$}
\input{convex_cctomcv2.tex}

\section{Moore's theorem and graphs on a surface}
\input{convex_moore.tex}
\section{$\mathrm{Unif}_\R$ implies $\mathrm{MCV}(3)$}
\input{convex_uniftomcv3.tex}

\section{Higher dimensions}
\input{convex_higher.tex}

\section{Summary}
\input{convex_summary.tex}
\section*{Acknowledgements}
\input{convex_acknowledgements.tex}
\bibliographystyle{plain}
\bibliography{locco}
\end{document}

%% file: convex_abstract.tex
Under $\mathrm{ZF}$, we show that the statement that every subset of every $\R$-vector space has a maximal convex subset is equivalent to the Axiom of Choice. We also study the strength of the same statement restricted to some specific $\R$-vector spaces. In particular, we show that the statement for $\R^2$ is equivalent to the Axiom of Countable Choice for reals, whereas the statement for $\R^3$ is equivalent to the Axiom of Uniformization. We discuss the statement for some spaces of higher dimensions as well.

%% file: convex_intro.tex
Convex geometry is a subject of mathematics which has relatively heavy dependence on the axiom of choice ($\mathrm{AC}$). For example, some basic theorems of this subject, like the Hahn-Banach theorems and the Krein-Milman theorems need certain fragments of $\mathrm{AC}$ in their proofs (see Bell-Fremlin \cite{bellfremlin72:_geometric} for the extent of dependence of these theorems on $\mathrm{AC}$).

In this paper, we investigate another, somewhat more naive, connection between convex sets and $\mathrm{AC}$. 

Let us consider the following family of statements:

\begin{dfn}\label{dfn:mcv}
For an $\R$-vector space $V$, $\mathrm{MCV}(V)$ denotes the statement that every subset $X$ of $V$ has a maximal convex subset. $\mathrm{MCV}$ denotes the statement that $\mathrm{MCV}(V)$ holds for every $\R$-vector space $V$. For a (not necessarily well-orderable) cardinal $\kappa$ we abusively write $\mathrm{MCV}(\kappa)$ to denote $\mathrm{MCV}(V)$ for an $\R$-vector space $V$ with a basis of cardinality $\kappa$.
\end{dfn}

Note that $\mathrm{MCV}$ can be proved by a typical use of Zorn's Lemma, an equivalent form of the Axiom of Choice ($\mathrm{AC}$). Therefore, under $\mathrm{ZF}$, $\mathrm{MCV}$ and $\mathrm{MCV}(V)$ for each $V$ can be regarded as fragments of $\mathrm{AC}$. In this paper we study the strengths of these statements under $\mathrm{ZF}$. 

This paper is organized as follows: The rest of this section we give basic notations and preliminaries used in this paper. The preliminaries contain a list of some relevant (mostly known) fragments of $\mathrm{AC}$, and a quick review on the basic theory of faces of convex sets in $\R$-vector spaces of finite dimensions.
In Section 2 we observe basic properties of $\mathrm{MCV}$ and $\mathrm{MCV}(V)$.
From Section 3 to 7 we will discuss $\mathrm{MCV}(2)$ and $\mathrm{MCV}(3)$.
In Section 3, we show that $\mathrm{MCV}(2)$ implies the Axiom of Countable Choice for reals ($\mathrm{CC}_\R$), and that $\mathrm{MCV}(3)$ implies the Axiom of Uniformization ($\mathrm{Unif}_\R$).
In Section 4 we prepare a general framework to find a maximal convex subset of a given subset of an $\R$-vector space of finite dimension, using the theory of faces.
In Section 5 we show that $\mathrm{CC}_\R$ implies $\mathrm{MCV}(2)$, using the framework constructed in Section 4.
In Section 6 we reproduce the proof of a (variation of) theorem concerning a topology of $\R^2$, which was proved by Moore \cite{moore1928:_triods}, only with a very weak fragment of $\mathrm{AC}$, and observe a lemma concerning planar graphs as an application of the theorem.
In Section 7 we show that $\mathrm{Unif}_\R$ implies $\mathrm{MCV}(3)$, again using the framework  of Section 4 and the lemma proved in Section 6.
In Section 8 we discuss $\mathrm{MCV}(V)$ for $V$'s of higher (infinite) dimensions.

Throughout this paper, unless otherwise stated we work in $\mathrm{ZF}$.

%% file: convex_notations.tex
\subsection{Notations and definitions}
Let $V$ be an $\R$-vector space.

For $\mathbf{p}$, $\mathbf{q}\in V$ we denote
\begin{eqnarray*}
(\mathbf{p}, \mathbf{q})&=&\{t\mathbf{p}+(1-t)\mathbf{q}\mid 0<t<1\}\ \text{and}\\
\lbrack \mathbf{p}, \mathbf{q} \rbrack&=&\{t\mathbf{p}+(1-t)\mathbf{q}\mid 0\leq t\leq1\}.
\end{eqnarray*}
Note that for distinct $\mathbf{p}$ and $\mathbf{q}$, $(\mathbf{p}, \mathbf{q})$ and $[\mathbf{p}, \mathbf{q}]$ respectively denote the open and closed line segments with endpoints $\mathbf{p}$ and $\mathbf{q}$.

Recall that $C\subseteq V$  is {\it convex\/} if $(\mathbf{p}, \mathbf{q})\in C$ for every $\mathbf{p}$, $\mathbf{q}\in C$.

For notations below, subscripts are often omitted if they are clear from the context.

For $X\subseteq V$, $\operatorname{conv}_V X$ denotes the {\it convex hull\/} of $X$, the smallest convex subset of $V$ containing $X$.
$\operatorname{aff}_V X$ denotes the {\it affine hull\/} of $X$, the smallest affine subspace of $V$ containing $X$. 

Now suppose $V=\R^n$ for some $n<\omega$. We consider $V$ as metrized and thus topologized by the standard Euclidean metric.

For $\mathbf{p}\in V$ and $r>0$, $B_V(\mathbf{p}; r)$ denotes the open ball of radius $r$ with center $\mathbf{p}$. 

For $X\subseteq V$, $\operatorname{cl}_V X$, $\operatorname{int}_V X$, $\partial_V X$ respectively denotes the closure, the interior and 
the boundary of $X$. 

$\dim X$ denotes $\dim\operatorname{aff}_V X$. 

$\operatorname{rint}_V X$ denotes the {\it relative interior\/} of $X$, that is, the interior of $X$ within $\operatorname{aff}_V(X)$.

$\operatorname{rbd}_V X$ denotes the {\it relative boundary\/} of $X$, that is, $\operatorname{rbd}_V X=\operatorname{cl}_V X\setminus\operatorname{rint}_V X$.

For a set $I$, we denote
\begin{eqnarray*}
&&\R^I=\{\mathbf{p}:I\to\R\}.\\
&&\R_I=\{\mathbf{p}\in\R^I\mid\text{$\{i\in I\mid\mathbf{p}(i)\not=0\}$ is finite}\}.
\end{eqnarray*}
Both $\R^I$ and $\R_I$ are naturally regarded as $\R$-vector spaces. 
For each $i\in I$, let $\mathbf{b}_i$ denote the member of $\R_I$ such that $\mathbf{b}_i(i)=1$ and $\mathbf{b}_i(j)=0$ (for every $j\in I$ with $j\not=i$). Then $\R_I$ is an $\R$-vector space with basis $\{\mathbf{b}_i\mid i\in I\}$.

%% file: convex_faces.tex
\subsection{Faces of convex sets}

Here we quickly review some basic facts in the theory of (extreme) faces of convex sets in finite dimensional $\R$-vector spaces. We will use these facts to develop a general method to find a maximal convex subset of a given subset of a finite dimensional $\R$-vector space.
Facts mentioned here are standard, or are easily deduced from standard facts. The readers are referred to Soltan \cite[Chapter 11]{soltan2020:_convexsets} for more reading on the theory of extreme faces.

Let $V$ be a finite dimensional $\R$-vector space.

\begin{dfn}\label{dfn:face}
Let $C\subseteq V$ convex. A subset $F\subseteq C$ is said to be a {\it face} (more precisely, an {\it extreme face\/}) of $C$ if
\begin{enumerate}[(i)]
\item $F$ is convex, and
\item whenever $\mathbf{p}$, $\mathbf{q}\in C$ and $(\mathbf{p}, \mathbf{q})\cap F\not=\emptyset$, then $\mathbf{p}$, $\mathbf{q}\in F$.
\end{enumerate}
Clearly $C$ and $\emptyset$ are faces of $C$. We call other faces as {\it proper faces\/} of $C$. We call a face of dimension $r$ as an {\it $r$-face}. We let $\mathcal{F}_C$ denote the set of faces of $C$.
\end{dfn}

The following are the basic facts about faces.

\begin{prop}\label{prop:basic}\normalfont
Let $C\subseteq V$ be convex.
\begin{enumerate}[(1)]
\item\label{item:faceofface} For any $F\in\mathcal{F}_C$ and $F'\subseteq F$, $F'\in\mathcal{F}_F$ iff $F'\in\mathcal{F}_C$.
\item\label{item:properface} If $F$ is a proper face of $C$, then $F\subseteq\operatorname{rbd}C$ and $\operatorname{dim}F<\operatorname{dim}C$.
\item\label{item:zeroface} If $\operatorname{dim} C=1$, $C$ has at most two $0$-faces.
\item\label{item:minusone} If $\operatorname{dim} C=r>1$, $C$ has at most countably many $(r-1)$-faces.
\end{enumerate}
\end{prop}



\begin{prop}\label{prop:face}
Let $C\subseteq V$ be convex.
\begin{enumerate}[(1)]
\item\label{item:lattice} $(\mathcal{F}_C, \subseteq)$ is a complete lattice with the maximum element $C$ and the minimum element $\emptyset$.
\item\label{item:rintunion} $C=\coprod_{F\in\mathcal{F}_C}\mathrm{rint} F=\coprod_{F\in\mathcal{F}_C\setminus\{\emptyset\}}\mathrm{rint} F$.
\item\label{item:pqrint} For every $F_0$, $F_1\in\mathcal{F}_C$ and distinct $\mathbf{p}$, $\mathbf{q}$ such that $\mathbf{p}\in\operatorname{rint} F_0$ and $\mathbf{q}\in\operatorname{rint} F_1$, $(\mathbf{p}, \mathbf{q})\subseteq\operatorname{rint} (F_0\lor F_1)$, where $F_0\lor F_1$ denotes the join of $F_0$ and $F_1$ in the lattice $(\mathcal{F}_C, \subseteq)$.
\end{enumerate}
\end{prop}

%% file: convex_choices.tex
\subsection{Fragments of $\mathrm{AC}$}

Here we list some relevant fragments of $\mathrm{AC}$. See \cite{howardrubin:_consequence}, \cite{herrlich:_choice} and \cite{jech:_choice} for more reading about fragments of $\mathrm{AC}$. 
\begin{dfn}\label{dfn:choice}
Let $X$ be a set.
\begin{enumerate}[(1)]
\item $\mathrm{WO}_X$ denotes the statement that $X$ is well-orderable.
\item $\mathrm{SC}_X$ denotes the statement that for every $C\subseteq[X]^{<\omega}$ satisfying $\emptyset\in C$ there exists a maximal $P\subseteq X$ such that $[P]^{<\omega}\subseteq C$.
\item $\mathrm{CQ}_X$ denotes the statement that every graph $G$ on $X$ (that is, $G$ is a subset of $[X]^2$) has a maximal clique (that is, a maximal $C\subseteq X$ such that $[C]^2\subseteq G$).
\item $\mathrm{EQ}_X$ denotes the statement that every equivalence relation $E$ on $X$ has a complete system of representatives.
\item $\mathrm{Unif}_X$ denotes the axiom of {\it uniformization on $X$}, which states that every binary relation on $X$ is uniformizable.
\item $\mathrm{DC}_X$ denotes the axiom of {\it dependent choice on $X$}, which states that whenever $R$ is a binary relation on $X$ such that for every $x\in X$ there exists $y\in X$ such that $xRy$, there exists a function $f:\omega\to X$ such that $f(n)Rf(n+1)$ for every $n<\omega$.
\item $\mathrm{CC}_X$ denotes the axiom of {\it countable choice on $X$}, which states that for every sequence $\seq{A_n\mid n<\omega}$ of nonempty subsets of $X$ there exists a function $f:\omega\to X$ such that $f(n)\in A_n$ for every $n<\omega$.
\end{enumerate}
\end{dfn}

We are mostly interested in case $X=\R$ or $\mathcal{P}(\R)$ for above fragments of $\mathrm{AC}$.
The following are easy implications between above fragments of $\mathrm{AC}$.
\begin{prop}\label{prop:easy}
\begin{enumerate}[(1)]
\item\label{item:yx} For any set $X$ and $Y$ such that $|X|\leq|Y|$, $\mathrm{WO}_Y$ ({\it resp.\/} $\mathrm{SC}_Y$, $\mathrm{CQ}_Y$, $\mathrm{EQ}_Y$, $\mathrm{Unif}_Y$, $\mathrm{DC}_Y$, $\mathrm{CC}_Y$) implies $\mathrm{WO}_X$ ({\it resp.\/} $\mathrm{SC}_X$, $\mathrm{CQ}_X$, $\mathrm{EQ}_X$, $\mathrm{Unif}_X$, $\mathrm{DC}_X$, $\mathrm{CC}_X$).
\item\label{item:wosccqeq} For any set $X$, for any two of $\mathrm{WO}_X$, $\mathrm{SC}_X$, $\mathrm{CQ}_X$, $\mathrm{EQ}_X$ the former implies the latter, and $\mathrm{Unif}_X$ implies $\mathrm{DC}_X$.
\item\label{item:equnif} For any set $X$, if $|X|^2=|X|$, then $\mathrm{EQ}_X$ implies $\mathrm{Unif}_X$.
\item\label{item:dccc}(Andretta-Notaro \cite{andrettanotaro25:_doesdc}) For any set $X$, if $|X\times 2|=|X|$, then $\mathrm{DC}_X$ implies $\mathrm{CC}_X$.
\item\label{item:unifwo} For any set $X$, $\mathrm{Unif}_{\mathcal{P}(X)}$ implies $\mathrm{WO}_X$.
\item\label{item:scult} For any set $X$, $\mathrm{SC}_{\mathcal{P}(X)}$ implies that every proper filter on $X$ can be extended to an ultrafilter.
\item\label{item:eqnonmeas} $\mathrm{EQ}_\R$ implies that $\omega_1\leq 2^\omega$ and the existence of Lebesgue non-measurable subsets of $\R$.
\end{enumerate}
\end{prop}

\proof (\ref{item:yx}) and (\ref{item:wosccqeq}) are straightforward.

\noindent (\ref{item:equnif}) Assume $\mathrm{EQ}_X$ and $|X|^2=|X|$. Then we have $\mathrm{EQ}_{X^2}$. Let $R$ be a binary relation on $X$. Define a binary relation $\equiv_R$ on $X^2$ by
$$
\seq{a, b}\equiv_R\seq{c, d}\Leftrightarrow a=c\land(b=d\lor b, d\in R).
$$
Then $\equiv_R$ is an equivalence relation on $X^2$. Thus by $\mathrm{EQ}_{X^2}$, we have a complete system $S$ of representatives, which essentially gives a uniformization of $R$.

\noindent For a proof of (\ref{item:dccc}) see \cite{andrettanotaro25:_doesdc}.

\noindent (\ref{item:unifwo}) Assume $\mathrm{Unif}_{\mathcal{P}(X)}$ and let
$$
R=\{\seq{A, B}\in(\mathcal{P}(X))^2\mid B\subseteq A\land\text{$B$ is a singleton}\}.
$$
Then a uniformization of $R$ essentially gives a choice function for $\mathcal{P}(X)\setminus\{\emptyset\}$.

\noindent (\ref{item:scult}) Assume $\mathrm{SC}_{\mathcal{P}(X)}$ and let $\mathcal{F}$ be any proper filter on $X$. Let
$$
\mathcal{C}=\{C\in[\mathcal{P}(X)]^{<\omega}\mid X\setminus\bigcap C\notin\mathcal{F}\}.
$$
Apply $\mathrm{SC}_{\mathcal{P}(X)}$ to obtain a maximal $\mathcal{G}\subseteq \mathcal{P}(X)$ such that $[\mathcal{G}]^{<\omega}\subseteq\mathcal{C}$. Then it is easy to see that $\mathcal{G}$ is an ultrafilter on $X$ extending $\mathcal{F}$.

\noindent (\ref{item:eqnonmeas}) Assume $\mathrm{EQ}_\R$. Then we have a right inverse to a surjection from $\R$ to $\omega_1$, which is an injection from $\omega_1$ to $\R$. Also by $\mathrm{EQ}_\R$ we have a right inverse to the canonical surjection from $\R$ to $\R/\Q$, and we can easily arrange it to a Vitali set\footnote{In fact, Shelah proved $2^\omega\geq\omega_1$ implies the existence of Lebesgue non-measurable subsets of $\R$ (see \cite{moore:_ac}).}.\qed

Some results on separating above fragments are known, or easily seen. We review some of them below.

\begin{prop}\label{prop:separation}
\begin{enumerate}[(1)]
\item\label{item:ccnontrivial} $\mathrm{ZF}$ does not imply $\mathrm{CC}_\R$.
\item\label{item:ccdc} $\mathrm{ZF}+\mathrm{CC}_\R$ does not imply $\mathrm{DC}_\R$.
\item\label{item:dcunif} $\mathrm{ZF}+\mathrm{DC}_\R$ does not imply $\mathrm{Unif}_\R$.
\item\label{item:unifeq} (Relative to the consistency of a large cardinal axiom) $\mathrm{ZF}+\mathrm{Unif}_\R$ does not imply $\mathrm{EQ}_\R$.
\item\label{item:wocc} $\mathrm{ZF}+\mathrm{WO}_\R$ does not imply $\mathrm{CC}_{\mathcal{P}(\R)}$.
\end{enumerate}
\end{prop}

\proof (\ref{item:ccnontrivial}) is due to Cohen \cite{cohen6364:_indep}. (\ref{item:ccdc}) is due to Jensen \cite{jensen67:_consistency} (see also Friedman, Gitman and Kanovei \cite{friedman19:_models}).

(\ref{item:dcunif}) On the one hand, assuming $\mathrm{ZF}+\mathrm{DC}_\R$ in the universe, it is easy to see that $\mathrm{DC}_\R$ holds in $L(\R)$. On the other hand, Solovay \cite{solovay:_independence} proved (under $\mathrm{ZF}$) that $\mathrm{Unif}_\R$ is equivalent to $\mathrm{AC}$ in $L(\R)$. There are several models of $\mathrm{ZFC}$ where $\mathrm{AC}$ fails in $L(\R)$ (including Solovay's model (\cite{solovay:_lebesgue}) obtained by collapsing an inaccessible cardinal to $\omega_1$), and thus $\mathrm{DC}_\R+\neg\mathrm{Unif}_\R$ holds in $L(\R)$. 

(\ref{item:unifeq}) $\mathrm{Unif}_\R$ is a consequence of the axiom $\mathrm{AD}_\R$, which is a strengthening of the axiom of determinacy ($\mathrm{AD}$). It is known that $\mathrm{ZF}+\mathrm{AD}_\R$ is consistent relative to a large cardinal axiom. Since $\mathrm{EQ}_\R$ implies the existence of a Lebesgue non-measurable subset of $\R$, which is inconsistent with $\mathrm{AD}$.

(\ref{item:wocc}) This separation can be observed by a standard use of a symmetric model (for symmetric models see \cite{jech:_choice}). Suppose $\mathrm{ZFC}+\mathrm{CH}$ in the ground model and let $\mathbb{P}$ be the poset consisting of countable partial functions from $\omega\times\omega_1\times\omega_1$ to $2$ ordered by reverse inclusion. Let $\mathcal{G}$ be the group consisting of $\omega$-sequences $\pi=\seq{\pi_n\mid n<\omega}$ of bijections from $\omega_1$ to $\omega_1$. For $\pi$, $\sigma\in\mathcal{G}$ we define $\pi\cdot\sigma=\seq{\pi_n\circ\sigma_n\mid n<\omega}$. Define an action of $\mathcal{G}$ on $\mathbb{P}$ as follows: For $\pi\in\mathcal{G}$ and $p\in\mathbb{P}$ let
$$
\pi(p)=\{\seq{n, \pi_n(\alpha), \beta, i}\in\omega\times\omega_1\times\omega_1\times 2\mid\seq{n, \alpha, \beta, i}\in p\}.
$$
Let $\mathcal{F}$ be the normal filter of subgroups of $\mathcal{G}$ generated by subgroups
$$
\mathcal{G}_n=\{\pi\in\mathcal{G}\mid\pi_n=\mathrm{id}\}\ (\text{for $n<\omega$}).
$$
Now let $N$ be a symmetric extension of the ground model by $\mathbb{P}$, $\mathcal{G}$ and $\mathcal{F}$. Since no new real is added, $\R$ is well-ordered and $2^\omega=\omega_1$ holds in $N$. On the other hand, it is easy to see that this extension naturally adds a countable family of sets of subsets of $\omega_1$ which has no choice function in $N$. Since $2^\omega=\omega_1$ it implies that $\mathrm{CC}_{\mathcal{P}(\R)}$ fails in $N$.\qed

Now we briefly mention sets encodable by reals. We make distinction between codings by which each object has a unique real code and those by which each object may have several real codes. Since we mostly work without $\mathrm{AC}$ in this paper, this distinction makes sense.

\begin{dfn}\label{dfn:encodable}
\begin{enumerate}[(1)]
\item We say a set $A$ is {\it $\R$-encodable\/} if $|A|\leq 2^\omega$, that is, there exists an injection  $A\to\R$.
\item We say a set $A$ is {\it $\R$-semiencodable\/} if there exists a surjection $\R\to A$.
\end{enumerate}
\end{dfn}

\begin{prop}\label{prop:examples}
\begin{enumerate}[(1)]
\item Every nonempty $\R$-encodable set is $\R$-semiencodable.
\item Under $\mathrm{ZF}$, the following sets are $\R$-encodable for each $n<\omega$:
\begin{itemize}
\item $\R^n$.
\item The set of finte subsets of $\R^n$.
\item The set of open or closed subsets of $\R^n$.
\item The set of affine subspaces of $\R^n$.
\item The set of convex subsets of $\R^n$ of dimension $\leq 1$.
\item The set of $\omega$-sequences of points in $\R^n$.
\end{itemize}
\item Under $\mathrm{ZF}$, the following sets are $\R$-semiencodable:
\begin{itemize}
\item The set of countable subsets of $\R^n$.
\item The set of $F_\sigma$ or $G_\delta$ subsets of $\R^n$.
\item $\omega_1$.
\end{itemize}
\end{enumerate}
\end{prop}

Using the notions of encodability and semiencodability, we can rephrase some fragments of $\mathrm{AC}$.

\begin{prop}
\begin{enumerate}[(1)]
\item $\mathrm{CC}_\R$ is equivalent to the statement that for any sequence $\seq{A_n\mid n<\omega}$ of nonempty subsets of an $\R$-semiencodable set $A$ there exists $f:\omega\to A$ such that $f(n)\in A_n$ for every $n<\omega$.
\item $\mathrm{Unif}_\R$ is equivalnet to the statement that for any $\R$-encodable set $I$ and any indexed family $\seq{A_i\mid i\in I}$ of nonempty subsets of an $\R$-semiencodable set $A$ there exists $f:I\to A$ such that $f(i)\in A_i$ for every $i\in I$.
\item $\mathrm{EQ}_\R$ is equivalent to the statement that for any $\R$-semiencodable set $I$ and any indexed family $\seq{A_i\mid
i\in I}$ of nonempty subsets of an $\R$-semiencodable set $A$ there exists $f:I\to A$ such that $f(i)\in A_i$ for every $i\in I$.\end{enumerate}
\end{prop}

%% file: convex_basics.tex
In this section we observe some basic facts about $\mathrm{MCV}$ and $\mathrm{MCV}(V)$ for particular $V$'s.

At first we have the following observation, roughly saying that $\mathrm{MCV}(V)$ for larger spaces are stronger than those for smaller spaces. Note that every image or inverse image of a convex set by an $\R$-linear map is convex.
\begin{prop}\label{prop:vw}
Let $V$, $W$ be $\R$-vector spaces.
\begin{enumerate}[(1)]
\item\label{item:inj} If there exists an injective $\R$-linear map $\iota:V\to W$, then $\mathrm{MCV}(W)$ implies $\mathrm{MCV}(V)$.
\item\label{item:surj} If there exists a surjective $\R$-linear map $\sigma:V\to W$, then $\mathrm{MCV}(V)$ implies $\mathrm{MCV}(W)$.
\end{enumerate}
\end{prop}

\proof (\ref{item:inj}) Suppose $\iota:V\to W$ is an injective $\R$-linear map and $\mathrm{MCV}(W)$ holds. Let $X\subseteq V$ be arbitrary. Then by $\mathrm{MCV}(W)$ there exists a maximal convex subset $C$ of $\iota''X\subseteq W$. Then let $D=\iota^{-1}C$. $D$ is convex, and since $\iota$ is injective, $D\subseteq X$ holds. If $D\subseteq D'\subseteq X$ and $D'$ is convex, then $C\subseteq\iota''D'\subseteq\iota''X$ and $\iota''D'$ is convex, and thus by maximality  $\iota''D'=C$. Since $\iota$ is injective $D'=D$ holds. This shows that $D$ is a maximal convex subset of $X$.

(\ref{item:surj}) Suppose $\sigma:V\to W$ is a surjective $\R$-linear map and $\mathrm{MCV}(V)$ holds. Let $X\subseteq W$ be arbitrary. Then by $\mathrm{MCV}(V)$ there exists a maximal convex subset $C$ of $\iota^{-1}X\subseteq V$. Then let $D=\iota''C$. $D$ is convex, and $D\subseteq X$ holds. If $D\subseteq D'\subseteq X$ and $D'$ is convex, then $C\subseteq\iota^{-1}D'\subseteq\iota^{-1}X$ and $\iota^{-1}D$ is convex, and thus by maximality $\sigma^{-1}D'=C$. Since $\sigma$ is surjective $D'=D$ holds. This shows that $D$ is a maximal convex subset of $X$.\qed

Now we give another easy observation about $\mathrm{MCV}(V)$.

\begin{prop}\label{prop:well}
Suppose $V$ is an $\R$-vector space. Then $\mathrm{WO}_V$ implies $\mathrm{MCV}(V)$.
\end{prop}

\proof Suppose $V$ is a well-ordered $\R$-vector space and $X\subseteq V$. Pick a well-ordered enumeration $\{\mathbf{p}_\gamma\mid\gamma<\lambda\}$ of $X$. We will define $\seq{C_\gamma\mid\gamma\leq\lambda}$ by induction on $\gamma$ as follows: Let $C_0=\emptyset$. For $\gamma<\lambda$ let
$$
\begin{cases}
\operatorname{conv}_V(C_\gamma\cup\{\mathbf{p}_\gamma\}) &\text{if $\operatorname{conv}_V(C_\gamma\cup\{\mathbf{p}_\gamma\})\subseteq X$,}\\
C_\gamma &\text{otherwise.}
\end{cases}
$$
And let $C_\gamma=\bigcup_{\xi<\gamma}C_\xi$ for limit $\gamma\leq\lambda$. Then it is easy to see that $C_\lambda$ is a maximal convex subset of $X$.\qed

\begin{cor}\label{cor:wor}
$\mathrm{WO}_{\R}$ implies $\mathrm{MCV}(\kappa)$ for every well-orderable cardinal $\kappa$.\qed
\end{cor}

%% file: convex_generalchoice.tex
\smallskip
Here we give a useful lemma to design an appropriate subset of an $\R$-vector space so that a maximal convex subset of it gives a choice function for a given family of sets.

\begin{lma}\label{lma:generalchoice}
Let $V$ be an $\R$-vector space. Suppose $X\subseteq V$ is of the form
$$
X=D\cup\bigcup_{i\in I}B_i
$$
where $D$ and $B_i$ ($i\in I$) are nonempty pairwise disjoint sets and satisfies the following:
\begin{enumerate}[(1)]
\item\label{item:bd} For every $\mathbf{p}\in X$ and $\mathbf{q}\in D$, $(\mathbf{p}, \mathbf{q})\subseteq D$ holds.
\item\label{item:sameb} For every two distinct $\mathbf{p}$ and $\mathbf{q}\in B_i$ for $i\in I$, $(\mathbf{p}, \mathbf{q})\nsubseteq X$ holds.
\item\label{item:distb} For every $\mathbf{p}\in B_i$ and $\mathbf{q}\in B_j$ for distinct $i$, $j\in I$, $(\mathbf{p}, \mathbf{q})\subseteq D$ holds.
\end{enumerate}
Then whenever $W$ is a maximal convex subset of $X$, then $D\subseteq W$ and $W\cap B_i$ contains exactly one point for every $i\in I$.
\end{lma}

\proof Suppose $W$ is a maximal convex subset of $X$. (\ref{item:bd}) assures that $W\cup D$ is convex. Thus by maximality $D\subseteq W$ holds. For each $i\in I$, by (\ref{item:sameb}), $W$ cannot contain two distinct points of $B_i$. On the other hand, suppose $W\cap B_i$ is empty for some $i\in I$. Pick $\mathbf{p}\in B_i$. Then by (\ref{item:bd}) and (\ref{item:distb}),  for every $\mathbf{q}\in W$ ite holds that $(\mathbf{p}, \mathbf{q})\subseteq D\subseteq W$, and thus $W\cup\{\mathbf{p}\}$ is convex, which contradicts the maximality of $W$. Therefore $|W\cap B_i|=1$ for every $i\in I$.\qed

%% file: convex_fullmcv.tex
\smallskip
As a first application of Lemma \ref{lma:generalchoice} we have the following:

\begin{thm}\label{thm:mcvac}
$\mathrm{MCV}$ implies $\mathrm{AC}$ (and therefore $\mathrm{MCV}$ is equivalent to $\mathrm{AC}$).
\end{thm}

\proof Let $\mathcal{X}$ be any set of nonempty sets. We will find a choice function for $\mathcal{X}$ using $\mathrm{MCV}$. Let $\mathcal{B}=\{\seq{A, a}\mid A\in\mathcal{X}\land a\in A\}$ and let $V=\R_\mathcal{B}$.

Let
$$
D=\{\mathbf{p}\in V\mid \mathbf{p}\geq 0 \land |\{A\in\mathcal{X}\mid \exists a\in A[\mathbf{p}(\seq{A, a})>0]\}|\geq 2\}.
$$
Let $X=D\cup\{\mathbf{b}_{\seq{A, a}}\mid \seq{A, a}\in \mathcal{B}\}$. It is easy to see that
\begin{enumerate}[(1)]
\item\label{item:XD} For every $\mathbf{p}\in X$ and $\mathbf{q}\in D$, $(\mathbf{p}, \mathbf{q})\subseteq D$.
\item\label{item:AA} For every $A\in\mathcal{X}$ and $a, a'\in A$ with $a\not=a'$, $\displaystyle\frac{1}{2}({\mathbf{b}_{\seq{A, a}}+\mathbf{b}_{\seq{A, a'}}})\notin X$.
\item\label{item:AA'} For every $\seq{A, a}, \seq{A', a'}\in \mathcal{B}$ with $A\not=A'$, $(\mathbf{b}_{\seq{A, a}}, \mathbf{b}_{\seq{A', a'}})\subseteq D$.

\end{enumerate}
Now apply $\mathrm{MCV}$ to get a maximal convex subset $C$ of $X$. Then by Lemma \ref{lma:generalchoice}, $C$ contains exactly one element of the form $\mathbf{b}_{\seq{A, a}}$ for each $A\in\mathcal{X}$. This shows that $C$ essentially gives a choice function for $\mathcal{X}$.\qed

%
%

%% file: convex_2and3.tex
From now on, we will discuss $\mathrm{MCV}(V)$ for $V$'s  of finite dimensions.

$\mathrm{MCV}(0)$ is trivial. $\mathrm{MCV}(1)$ is also easily provable under $\mathrm{ZF}$, since for every nonempty subset $X\subseteq\R$, each connected component $C$ of $X$ is a maximal convex subset of $X$.

Then let us consider about $\mathrm{MCV}(2)$ and $\mathrm{MCV}(3)$. In the rest of this section we observe implications from $\mathrm{MCV}(2)$ and $\mathrm{MCV}(3)$, using Lemma \ref{lma:generalchoice}. In later sections we discuss implications of the other direction.

\begin{thm}[$\mathrm{ZF}$]\label{thm:2and3}
\begin{enumerate}[(1)]
\item\label{item:2cc} $\mathrm{MCV}(2)$ implies $\mathrm{CC}_\R$.
\item\label{item:3unif} $\mathrm{MCV}(3)$ implies $\mathrm{Unif}_\R$.
\end{enumerate}
\end{thm}

\proof (\ref{item:2cc}) Fix a bijiection $\varphi:\R\to (0, 1)\setminus\Q$. For $n<\omega$, let $\mathbf{p}_n=\seq{n, n^2}\in\R^2$. Let $f:[0, \infty)\to\R$ be the function which has the polygonal line
$$
\bigcup_{n<\omega}[\mathbf{p}_n, \mathbf{p}_{n+1}]
$$
as its graph. For each $n<\omega$ let $\psi_n$ be the affine function which maps $(0, 1)$ onto $(\mathbf{p}_n, \mathbf{p}_{n+1})$. Now set
$$
D_2=\{\seq{x, y}\in\R^2\mid x>0\land y>f(x)\}.
$$
Note that $D_2$ is an open convex subset of $\R^2$ and each $[\mathbf{p}_n, \mathbf{p}_{n+1}]$ is an edge of $D_2$.
Now let $\seq{A_n\mid n<\omega}$ be an arbitrary sequence of nonempty sets of reals. For each $n<\omega$ let $B_n=(\psi_n\circ\varphi)''A_n$ and let
$$
X_2=D_2\cup\bigcup_{n<\omega}B_n.
$$
It is easy to see that
\begin{enumerate}[(1)]
\item\label{item:X2D2} For every $\mathbf{p}\in X_2$ and $\mathbf{q}\in D_2$, $(\mathbf{p}, \mathbf{q})\subseteq D_2$.
\item\label{item:mm} For every $m<\omega$ and distinct $\mathbf{p}, \mathbf{q}\in X_2\cap[\mathbf{p}_m, \mathbf{p}_{m+1}]$, $(\mathbf{p}, \mathbf{q})\nsubseteq X_2$.
\item\label{item:mn} For every $\mathbf{p}\in X_2\cap[\mathbf{p}_m, \mathbf{p}_{m+1}]$ and $\mathbf{q}\in X_2\cap[\mathbf{p}_n, \mathbf{p}_{n+1}]$ with $m\not=n$, $(\mathbf{p}, \mathbf{q})\subseteq D_2$.
\end{enumerate}
Now apply $\mathrm{MCV}(2)$ to obtain a maximal convex subset $C$ of $X$. Then by Lemma \ref{lma:generalchoice} $C$ contains exactly one point from $B_n$ for each $n<\omega$. This shows that $C$ essentially gives a choice function for the sequence $\seq{A_n\mid n<\omega}$.

\smallskip\noindent
(\ref{item:3unif}) Let $\varphi:\R\to(0, 1)\setminus \Q$ be as above and fix another bijection $\eta:\R\to [0, 2\pi)$. Let
$$
D_3=\{\seq{x, y, z}\in\R^3\mid x^2+y^2<1\land 0<z<1\}.
$$
Let $\seq{A_r\mid r\in\R}$ be an arbitrary sequence of nonempty sets of reals, and for each $r\in\R$ let
$$
B_r=\{\seq{\cos\eta(r), \sin\eta(r), z}\mid z\in\varphi''A_r\}.
$$
Note that $D_3$ is the interior of the body of a circular cylinder, and that each $B_r$ is on a generatrix of the cylinder.
Now let
$$
X_3=D_3\cup\bigcup_{r\in\R}B_r.
$$
Again it is easy to see that
\begin{enumerate}[(1)]
\item\label{item:X3D3} For every $\mathbf{p}\in X_3$ and $\mathbf{q}\in D_3$, $(\mathbf{p}, \mathbf{q})\subseteq D_3$.
\item\label{item:rr} For every $r\in\R$ and distinct $\mathbf{p}, \mathbf{q}\in X_3\cap B_r$, $(\mathbf{p}, \mathbf{q})\nsubseteq X_3$.
\item\label{item:rr'} For every $\mathbf{p}\in X_3\cap B_r$ and $\mathbf{q}\in X_3\cap B_{r'}$ with $r\not=r'$, $(\mathbf{p}, \mathbf{q})\subseteq D_3$.
\end{enumerate}
Apply $\mathrm{MCV}(3)$ to obtain a maximal convex subset $C$ of $X_2$. Again by Lemma \ref{lma:generalchoice}, $C$ essentially gives a choice function for the sequence $\seq{A_r\mid r\in\R}$.\qed

%% file: convex_generalapproach.tex
In this section we introduce a general framework to find out a maximal convex subset of a given subset of an $\R$-vector space, in particular of finite dimension.

\begin{dfn}\label{dfn:cvflt}
Let $V$ be an $\R$-vector space.
\begin{enumerate}[(1)]
  \item\label{item:cfc} $\vec{C}=\seq{C_\gamma\mid\gamma\leq\alpha}$ for an ordinal $\alpha$ is said to be a {\it convex filtration\/} (in $V$) if $\vec{C}$ is a $\subseteq$-continuous increasing sequence of convex subsets of $V$.
  \item\label{item:mcfc} A convex filtration $\vec{C}=\seq{C_\gamma\mid\gamma\leq\alpha}$ in $V$ is said to be {\it fine\/} with a {\it partition sequence\/} $\seq{\mathcal{P}_\gamma\mid\gamma<\alpha}$ if for every $\gamma<\alpha$
\begin{enumerate}[(i)]
  \item\label{item:mcfc0} $C_\gamma\cup\{\mathbf{p}\}$ is convex for every $\mathbf{p}\in C_{\gamma+1}\setminus C_\gamma$,
  \item\label{item:mcfc1} $\mathcal{P}_\gamma$ is a partition of $C_{\gamma+1}\setminus C_\gamma$ into convex sets, and
  \item\label{item:mcfc2} for every $\mathbf{p}\in E$ and $\mathbf{q}\in E'$ for distinct $E$, $E'\in\mathcal{P}_\gamma$, $(\mathbf{p}, \mathbf{q})\subseteq C_\gamma$ holds.
\end{enumerate}
\end{enumerate}
\end{dfn}


\begin{lma}\label{lma:general}
Suppose $X\subseteq V$ and $\vec{C}=\seq{C_\gamma\mid\gamma\leq\alpha}$ is a convex filtration such that $X\subseteq C_\alpha$.
\begin{enumerate}[(1)]
\item\label{item:gencfc} Suppose $\vec{W}=\seq{W_\gamma\mid\gamma\leq\alpha}$ satisfies the following:
\begin{enumerate}[(i)]
\item\label{item:general1} $W_0$ is a maximal convex subset of $X\cap C_0$.
\item\label{item:general2} For each $\gamma<\alpha$, $W_{\gamma+1}$ is a maximal $Z\subseteq X\cap C_{\gamma+1}$ such that $Z$ is convex and $Z\cap C_\gamma=W_\gamma$.
\item\label{item:general3} For each limit $\gamma\leq\alpha$, $W_\gamma=\bigcup\{W_\xi\mid\xi<\gamma\}$.
\end{enumerate}
Then $W_\alpha$ is a maximal convex subset of $X$.
\item\label{item:genmcfc} Suppose that $\vec{C}$ is fine with a partition sequence $\seq{\mathcal{P}_\gamma\mid\gamma<\alpha}$
and $\vec{W}=\seq{W_\gamma\mid\gamma\leq\alpha}$ satisfies (\ref{item:general1}), (\ref{item:general3}) above and 
the following, instead of (\ref{item:general2}):
\begin{enumerate}[(i)]
\setcounter{enumii}{3}
\item For each $\gamma<\alpha$, $W_{\gamma+1}$ is maximal $Z\subseteq X\cap C_{\gamma+1}$ such that:
\begin{enumerate}[(a)]
  \item\label{item:a} $Z\cap C_\gamma=W_\gamma$,
  \item\label{item:b} $W_\gamma\cup\{\mathbf{p}\}$ is convex for every $\mathbf{p}\in Z\setminus C_\gamma$,
  \item\label{item:c} $Z\cap E$ is convex for every $E\in\mathcal{P}_\gamma$ and
  \item\label{item:d} for every $\mathbf{p}\in Z\cap E$ and $\mathbf{q}\in Z\cap E'$ for distinct $E$, $E'\in\mathcal{P}_\gamma$, $(\mathbf{p}, \mathbf{q})\subseteq W_\gamma$ holds.
\end{enumerate}
\end{enumerate}
Then $W_\alpha$ is a maximal convex subset of $X$.
\end{enumerate}
\end{lma}


%

\proof (\ref{item:gencfc}) It is clear that $\vec{W}$ is $\subseteq$-increasing and continuous. $W_\gamma$ is convex for $\gamma=0$ or successor by (\ref{item:general1}) and (\ref{item:general2}). $W_\gamma$ is convex also for limit $\gamma$, since the union of a $\subseteq$-chain of convex sets is convex. This shows that $\vec{W}$ is a convex filtration. It is also easy to see by induction that $W_\beta\cap C_\gamma=W_\gamma$ for each $\gamma\leq\beta\leq\alpha$.

Now to see $W_\alpha$ is maximal, suppose $W'\supsetneq W_\alpha$ is a convex subset of $X$ (and thus of $C_\alpha$). Let $\gamma\leq\alpha$ be the least such that $W'\cap C_\gamma\supsetneq W_\gamma$. But $\gamma\not=0$ by (\ref{item:general1}), $\gamma$ cannot be successor by (\ref{item:general2}) and $\gamma$ cannot be limit by the continuity of $\vec{W}$. This is a contradiction.

\noindent
(\ref{item:genmcfc}) It is enough to show that for each $\gamma<\alpha$, assuming that $W_\gamma$ is convex, for each $Z\subseteq X\cap C_{\gamma+1}$ satisfying $Z\cap C_\gamma=W_\gamma$, $Z$ is convex if and only if $Z$ satisfies (b), (c) and (d).

Then let $\gamma<\alpha$ and assume $W_\gamma$ is convex. Let $Z$ be arbitrary such that $Z\subseteq X\cap C_{\gamma+1}$ and $Z\cap C_\gamma=W_\gamma$.

First suppose $Z$ is convex. For any $\mathbf{p}\in Z\setminus C_\gamma\subseteq C_{\gamma+1}\setminus C_\gamma$, by Definition \ref{dfn:cvflt}(\ref{item:mcfc})(\ref{item:mcfc0}), $C_\gamma\cup\{\mathbf{p}\}$ is convex, and thus $Z\cap(C_\gamma\cup\{\mathbf{p}\})=W_\gamma\cup\{\mathbf{p}\}$ is also convex, which shows (b). Since each $E\in\mathcal{P}_\gamma$ is convex by Definition \ref{dfn:cvflt}(\ref{item:mcfc})(\ref{item:mcfc1}), $Z\cap E$ is convex as well, which shows (c). Now suppose $\mathbf{p}\in Z\cap E$ and $\mathbf{q}\in Z\cap E'$ for distinct $E$, $E'\in\mathcal{P}_\gamma$. Then by Definition \ref{dfn:cvflt}(\ref{item:mcfc})(\ref{item:mcfc2}) and since $Z$ is convex, it holds that $(\mathbf{p}, \mathbf{q})\subseteq C_\gamma\cap Z=W_\gamma$, which shows (d).

To show the other direction suppose $Z$ satisfies (b), (c) and (d). Let $\mathbf{p}$, $\mathbf{q}\in Z$ be distinct. In case $\mathbf{p}$, $\mathbf{q}\in Z\cap C_\gamma=W_\gamma$, by our assumption that $W_\gamma$ is convex we have $(\mathbf{p}, \mathbf{q})\subseteq W_\gamma\subseteq Z$. In case $\mathbf{p}\in Z\cap C_\gamma$ and $\mathbf{q}\in Z\setminus C_\gamma$, 
$(\mathbf{p}, \mathbf{q})\subseteq Z$ follows from (b). In case $\mathbf{p}$, $\mathbf{q}\in Z\setminus C_\gamma$ and if both belong to $Z\cap E$ for the same $E\in\mathcal{P}_\gamma$, $(\mathbf{p}, \mathbf{q})\subseteq Z$ follows from (c), and if they belong respectively to $Z\cap E$ and $Z\cap E'$ for distinct $E$, $E'\in\mathcal{P}_\gamma$, $(\mathbf{p}, \mathbf{q})\subseteq Z$ follows from (d). This shows that $Z$ is convex.\qed

Lemma \ref{lma:general} shows how we may find a maximal convex subset of a given $X$: First choose an appropriate convex filtration $\vec{C}=\seq{C_\gamma\mid\gamma\leq\alpha}$ such that $X\subseteq C_\alpha$. Then choose a maximal convex subset $W_0$ of $X\cap C_0$ first, and once $W_\gamma$ is chosen then extend it to a convex subset $W_{\gamma+1}$ of $X\cap C_{\gamma+1}$ without adding points in $C_\gamma$ so that no such convex subset of $X\cap C_{\gamma+1}$ properly containing $W_{\gamma+1}$ exists. For limit $\gamma$ let $W_\gamma$ be the union of all preceding $W_\xi$'s. Then $W_\alpha$ will be a maximal convex subset of $X$. In case $\vec{C}$ is fine with a partition sequence $\seq{\mathcal{P}_\gamma\mid\gamma<\alpha}$, to choose $W_{\gamma+1}$ we may first let
$$
X'=\{\mathbf{p}\in X\cap (C_{\gamma+1}\setminus C_\gamma)\mid\text{$W_\gamma\cup\{\mathbf{p}\}$ is convex}\}
$$
and choose a convex $C_E\subseteq X'\cap E$ for each $E\in\mathcal{P}_\gamma$ so that for every distinct $E$, $E'\in\mathcal{P}_\gamma$, it holds that $(\mathbf{p}, \mathbf{q})\subseteq W_\gamma$ for every $\mathbf{p}\in E$ and $\mathbf{q}\in E'$, and also that whenever $\seq{C'_E\mid E\in\mathcal{P}_\gamma}$ satisfies the same conditions and $C'_E\supseteq C_E$ holds for each $E\in\mathcal{P}_\gamma$, it holds that $C'_E=C_E$ for every $E\in\mathcal{P}_\gamma$. Then we may let
$$
W_{\gamma+1}=W_\gamma\cup\bigcup_{E\in\mathcal{P}_\gamma}C_E.
$$

Next we introduce a method to construct a convex filtration using the structure of faces, in case $V$ is of finite dimension.

\begin{dfn}\label{dfn:faceflt}
Let $V$ be an $\R$-vector space of finite dimension and $C$ a convex subset of $V$.
\begin{enumerate}[(1)]
  \item We say $\vec{\mathcal{F}}=\seq{\mathcal{F}_\gamma\mid\gamma\leq\alpha}$ is a {\it face filtration\/} of $C$ if it is a $\subseteq$-continuous increasing sequence of upward closed subsets of $(\mathcal{F}_C\setminus\{\emptyset\}, \subseteq)$ such that $\mathcal{F}_\alpha=\mathcal{F}_C\setminus\{\emptyset\}$.
  \item We say a face filtration $\vec{\mathcal{F}}=\seq{\mathcal{F}_\gamma\mid\gamma\leq\alpha}$ of $C$ is {\it fine\/} if for each $\gamma<\alpha$ every two distinct members of $\mathcal{F}_{\gamma+1}\setminus\mathcal{F}_\gamma$ are $\subseteq$-incomparable.
\end{enumerate}
\end{dfn}

\begin{lma}\label{lma:face}
Let $V$ be an $\R$-vector space of finite dimension, $C$ a convex subset of $V$ and $\vec{\mathcal{F}}=\seq{\mathcal{F}_\gamma\mid\gamma\leq\alpha}$ a face filtration of $C$.
\begin{enumerate}[(1)]
  \item\label{item:facefiltration} For each $\gamma\leq\alpha$ let $$\mathcal{F}^\circ_\gamma=\bigcup_{F\in\mathcal{F}_\gamma}\operatorname{\mathrm{rint}}F.$$
Then $\vec{\mathcal{F}^\circ}=\seq{\mathcal{F}^\circ_\gamma\mid\gamma\leq\alpha}$ is a convex filtration satisfiying $\mathcal{F}^\circ_\alpha=C$.
  \item\label{item:finefacefiltration} If $\vec{\mathcal{F}}$ is fine, then by letting
$$\mathcal{P}_\gamma=\{\operatorname{rint}F\mid F\in\mathcal{F}_{\gamma+1}\setminus\mathcal{F}_\gamma\}$$
$\vec{\mathcal{F}^\circ}$ is a fine convex filtration with a partition sequence $\seq{\mathcal{P}_\gamma\mid\gamma<\alpha}$.
\end{enumerate}
\end{lma}

\proof (\ref{item:facefiltration}) Since each $\mathcal{F}_\gamma$ for $\gamma\leq\alpha$ is upward closed, by Proposition \ref{prop:face}(\ref{item:pqrint}) $\mathcal{F}^\circ_\gamma$ is convex. Thus it is clear that $\seq{\mathcal{F}^\circ_\gamma\mid\gamma\leq\alpha}$ is a convex filtration. $\mathcal{F}^\circ_\alpha=C$ follows from Proposition \ref{prop:face}(\ref{item:rintunion}).

\noindent
(\ref{item:finefacefiltration}) Suppose $\vec{\mathcal{F}}$ is fine and let $\seq{\mathcal{P}_\gamma\mid\gamma<\alpha}$ be as above. For each $\gamma<\alpha$ we will show (\ref{item:mcfc0})--(\ref{item:mcfc2}) of Definition \ref{dfn:cvflt}(\ref{item:mcfc}). (\ref{item:mcfc0}) follows from Proposition \ref{prop:face}(\ref{item:pqrint}) and our assumption that $\mathcal{F}_\gamma$ is upward closed. (\ref{item:mcfc1}) follows from Proposition \ref{prop:face}(\ref{item:rintunion}), since the relative interior of any convex set is convex. For every two distinct faces $F$, $F'\in\mathcal{F}_{\gamma+1}\setminus\mathcal{F}_\gamma$, since they are $\subseteq$-incomparable and $\mathcal{F}_{\gamma+1}$ is upward closed, $F\lor F'\in\mathcal{F}_{\gamma+1}$ holds. But since $F\subsetneq F\lor F'$, $F\lor F'$ cannot be a member of $\mathcal{F}_{\gamma+1}\setminus\mathcal{F}_\gamma$, and thus $F\lor F'\in\mathcal{F}_\gamma$ holds. Therefore we have (\ref{item:mcfc2}) by Proposition \ref{prop:face}(\ref{item:pqrint}) again. This completes the proof that $\vec{\mathcal{F}^\circ}$ is a fine convex filtration with $\seq{\mathcal{P}_\gamma\mid\gamma<\alpha}$.\qed

%

%% file: convex_cctomcv2.tex
In this section we prove the following, using the framework given in the previous section.
\begin{thm}\label{thm:cctotwo}
$\mathrm{CC}_\R$ implies $\mathrm{MCV}(2)$. 
\end{thm}

\proof Fix an enumeration $\seq{U_n\mid n<\omega}$ of open basis of $\R^2$. Fix also a total ordering $<_2$ of $\R^2$ (we may use the lexicographic order for example).
Suppose $X\subseteq\R^2$ and we will find a maximal convex subset of $X$. We may assume that $X$ is nonempty. Let
$$
k:=\max\{\dim C\mid\text{$C$ is a nonempty convex subset of $X$}\}.
$$
Note that $0\leq k\leq 2$. In case $k=0$, any point of $X$ is a maximal convex subset of $X$, since any convex set with more than two points would have positive dimension. In case $k=1$, pick a convex $C\subseteq X$ of dimension $1$, and let $L=\operatorname{aff} C$, that is the line in which $C$ lies. Let $C'$ be the connected component of $L\cap X$ containing $C$. Then $C'$ is a maximal convex subset of $X$, since $C'$ is clearly maximal within $L\cap X$, and any convex subset containing $C'$ together with a point not in $L$ would have dimension $2$.

Now we assume $k=2$. Let $D_0=\emptyset$ and define $D_{n+1}$ by induction on $n<\omega$ by
$$
D_{n+1}:=
\begin{cases}
\operatorname{conv}(D_n\cup U_n)&\text{if $\operatorname{conv}(D_n\cup U_n)\subseteq X$,}\\
D_n&\text{otherwise}.
\end{cases}
$$
Then let $D:=\bigcup_{n<\omega}D_n$. Note that $D\not=\emptyset$, since $X$ has a convex subset of dimension $2$, which must contain a nonempty open convex subset. Since the convex hull of an open subset of $\R^2$ is open, $D$ is a maximal open convex subset of $X$. We will find a maximal convex subset of $X$ by extending $D$. Let $C:=\operatorname{cl} D$. Since any convex set containing $D$ together with a point not in $C$ would have the convex interior which strictly extends $D$, by maximality of $D$ any convex subset of $X$ containing $D$ must be contained in $C$. Thus for our purpose we may assume that $D\subseteq X\subseteq C$.
Let
\begin{eqnarray*}
\mathcal{E}&=&\text{the set of $1$-faces of $C$,}\\
\mathcal{V}_1&=&\text{the set of $0$-faces of $C$ contained in some $1$-face of $C$,}\\
\mathcal{V}&=&\text{the set of other $0$-faces of $C$.}
\end{eqnarray*}

Note that $|\mathcal{E}|\leq\omega$  by Proposition \ref{prop:basic}(\ref{item:minusone}). Since each $1$-face of $C$ contains at most two $0$-faces of $C$ by Proposition \ref{prop:basic}(\ref{item:faceofface}) and (\ref{item:zeroface}), we also have $|\mathcal{V}_1|\leq\omega$ (Here we don't need any fragment of $\mathrm{AC}$, since using $<_2$ we can uniformly choose orderings of the $0$-faces contained in each $1$-face of $C$). Let $\seq{P_i=\{\mathbf{p}_i\}\mid i<m}$ ($m\leq\omega$) be an enumeration of $\mathcal{V}_1$. Now let
\begin{eqnarray*}
\mathcal{F}_0&=&\{C\},\\
\mathcal{F}_1&=&\mathcal{F}_0\cup\mathcal{E},\\
\mathcal{F}_2&=&\mathcal{F}_1\cup\mathcal{V},\\
\mathcal{F}_{i+3}&=&\mathcal{F}_{i+2}\cup\{P_i\}\quad(\text{for $i<m$}),\\
\mathcal{F}_\omega&=&\bigcup_{i<\omega}\mathcal{F}_i\quad(\text{if $m=\omega$}).
\end{eqnarray*}

Then $\seq{\mathcal{F}_i\mid i\leq 2+m}$ is a fine face filtration of $C$. Let $\seq{\mathcal{F}^\circ_i\mid i\leq 2+m}$ be the fine convex filtration derived as in Lemma \ref{lma:face}. We will find $\seq{W_i\mid i\leq 2+m}$ as in Lemma \ref{lma:general}(\ref{item:genmcfc}). Then $W_{2+m}$ will be a maximal convex subset of $X$. 

First of all, since any open convex set is regularly open we have $\mathcal{F}^\circ_0=\operatorname{rint}C=\operatorname{int}C=D\subseteq X$. Thus we may set $W_0=D$.

Now let us choose $W_1$. Note first that by fineness $W_0\cup\{\mathbf{p}\}=\mathcal{F}^\circ_0\cup\{\mathbf{p}\}$ is convex for each $\mathbf{p}\in\mathcal{F}^\circ_1\setminus\mathcal{F}^\circ_0$. Note also that $\mathcal{F}^\circ_1\setminus\mathcal{F}^\circ_0=\coprod_{E\in\mathcal{E}}\operatorname{rint}E$, and for any $\mathbf{p}\in\operatorname{rint} E$ and $\mathbf{q}\in\operatorname{rint} E'$ for distinct $E$, $E'\in\mathcal{E}$, $(\mathbf{p}, \mathbf{q})\subseteq D=W_0$ holds by Proposition \ref{prop:face}(\ref{item:pqrint}) (since $E\lor E'=C$ in the lattice $(\mathcal{F}_C, \subseteq)$). Therefore to find $W_1$ satisfying conditions in Lemma \ref{lma:general}(\ref{item:genmcfc}), we simply may choose a maximal convex subset $C_E$ of $X\cap\operatorname{rint}E$ for each $E\in\mathcal{E}$ and let $$W_1=W_0\cup\bigcup\{C_E\mid E\in\mathcal{E}\}.$$ For each $E\in\mathcal{E}$, if $X\cap\operatorname{rint}E=\emptyset$ we may let $C_E=\emptyset$. Otherwise, since $\operatorname{dim}(X\cap\operatorname{rint}E)\leq 1$, we may choose a connected component of $X\cap\operatorname{rint}E$ as $C_E$. Since such a component can be coded by a real, we can choose $\seq{C_E\mid E\in\mathcal{E}}$ by using $\mathrm{CC}_\R$.

Nextly let us consider about $W_2$. For each $\{\mathbf{p}\}\in\mathcal{V}$ and any other proper face $F$ of $C$ it holds that $\{\mathbf{p}\}\lor F=C$ in $(\mathcal{F}_C, \subseteq)$, and thus by Proposition \ref{prop:face}(\ref{item:pqrint}), for any $\mathbf{p}\in\bigcup\mathcal{V}$ and any other $\mathbf{q}\in C$ it holds that $(\mathbf{p}, \mathbf{q})\subseteq D\subseteq W_1$. This means that we may simply let $$W_2=W_1\cup(X\cap\bigcup\mathcal{V}).$$

Now suppose $W_{i+2}$ is already chosen for $i<m$. Then we may let
$$
W_{i+3}=
\begin{cases}
W_{i+2}\cup\{\mathbf{{p}_i}\}&\text{if $\mathbf{p}_i\in X$ and $W_{i+2}\cup\{\mathbf{{p}_i}\}$ is convex,}\\
W_{i+2}&\text{otherwise.}
\end{cases}
$$
Then if $m=\omega$ just let $W_\omega=\bigcup_{i<\omega}W_i$. This completes the construction of $\seq{W_i\mid i\leq 2+m}$ as desired.\qed

By Theorems \ref{thm:2and3}(\ref{item:2cc}) and \ref{thm:cctotwo} we have
\begin{cor}\label{cor:ccand2}
$\mathrm{MCV}(2)$ is equivalent to $\mathrm{CC}_\R$.
\end{cor}

Reflecting the proof of Theorem \ref{thm:cctotwo}, we have the following stronger statement, which we will use later.
\begin{cor}\label{cor:ccborel}
Assume $\mathrm{CC}_\R$. Then for every $X\subseteq\R^2$, there exists a $G_\delta$ subset $S\subseteq\R^2$ such that $X\cap S$ is a maximal convex subset of $X$.
\end{cor}

\proof Let $X\subseteq\R^2$ be arbitrary. We may assume $X\not=\emptyset$, and in case $k$ in the proof of Theorem \ref{thm:cctotwo} is $0$ or $1$, we found a maximal convex subset of $X$ which itself is $G_\delta$. So we may assume $k=2$. In this case the maximal convex subset of $X$ found in the proof of Theorem \ref{thm:cctotwo} can be written as $X\cap S$, where
$$
S=C\setminus(\bigcup_{E\in\mathcal{E}}(\operatorname{rint}E\setminus C_E))\setminus\{\mathbf{p}_i\mid i<m\land \mathbf{p}_i\notin W_{2+m}\},
$$
where $C$ is closed, $\mathcal{E}$ is at most countable, $\operatorname{rint}E$ and $C_E$ are convex sets of dimension $\leq 1$ (or empty) and $m\leq\omega$. Note that under $\mathrm{CC}_\R$ every countable intersection of $G_\delta$ subsets of $\R^2$ is $G_\delta$. Therefore $S$ is $G_\delta$.\qed

%% file: convex_moore.tex
Here we give a lemma we will use in the next section, which is a minor variation of a theorem proved by Moore \cite{moore1928:_triods}, which states that there can be at most countably many pairwise disjoint subsets of $\R^2$, each of which is homeomorphic to the union of three distinct closed line segments sharing one endpoint (the shape of letter Y). We will give a full proof of our lemma to make sure that we only need a weak fragment of $\mathrm{AC}$ to prove it, although the basic idea of the proof is due to \cite{moore1928:_triods}. To reduce the amount of our task, we restrict our claim to mention only figures consisting of line segments. Instead, we relax our assumption so that our figures may have more than three line segments. This generalization makes sense because we are restricting our use of $\mathrm{AC}$.

\begin{dfn}
Suppose $V$ is an $\R$-vector space.
A {\it star configuration\/} in $V$ is a pair $\seq{\mathbf{p}, A}$ satisfying $\mathbf{p}\in V$, $A\subseteq V\setminus\{\mathbf{p}\}$, $|A|\geq 3$, and $(\mathbf{p}, \mathbf{q})\cap(\mathbf{p}, \mathbf{r})=\emptyset$ for every two distinct $\mathbf{q}$, $\mathbf{r}\in A$. We call $\mathbf{p}$ as the {\it center\/} of the configuration $\seq{\mathbf{p}, A}$. For a star configuration $\seq{\mathbf{p}, A}$, the {\it star\/} $S_\seq{\mathbf{p}, A}$ of $\seq{\mathbf{p}, A}$ is defined by
$$
S_\seq{\mathbf{p}, A}=\bigcup_{\mathbf{q}\in A}[\mathbf{p}, \mathbf{q}].
$$
For a subset $W$ of $V$, we say $\seq{\mathbf{p}, A}$ is a star configuration in $W$ if $S_\seq{\mathbf{p}, A}\subseteq W$ holds.
\end{dfn}
\begin{thm}[A variation of a theorem of Moore \cite{moore1928:_triods}]\label{thm:moore} Assume $\mathrm{CC}_\R$\footnote{In fact, to prove this theorem we only use the consequence of $\mathrm{CC}_\R$ that every countable union of countable sets of reals is countable.}.
Suppose $\mathcal{C}$ is an uncountable family of star configurations in $\R^2$. Then $\mathcal{C}$ has two distinct members whose stars have nonempty intersection.
\end{thm}

\proof Suppose $\mathcal{C}$ is an uncountable family of star configurations in $\R^2$. We may assume that the centers of configurations in $\mathcal{C}$ are all distinct. Thus $\mathcal{C}$ is $\R$-encodable, and by $\mathrm{CC}_\R$, pick an uncountable $\mathcal{C}_0\subseteq\mathcal{C}$ and $k<\omega$ such that for every $\seq{\mathbf{p}, A}\in\mathcal{C}_0$, $A$ contains at least three points whose distance from $\mathbf{p}$ is larger than $\frac{1}{k}$. Again by $\mathrm{CC}_\R$, pick an uncountable $\mathcal{C}_1\subseteq\mathcal{C}_0$ and $\mathbf{p}_0\in \R^2$ such that every center of configuration in $\mathcal{C}_1$ is in $B(\mathbf{p}_0; \frac{1}{3k})$. Let $D=B(\mathbf{p}_0; \frac{2}{3k})$ and $C=\partial D$. For each configuration $\seq{\mathbf{p}, A}\in\mathcal{C}_1$, $S_\seq{\mathbf{p}, A}\cap C$ contains at least three points. By $\mathrm{CC}_\R$, pick an uncountable $\mathcal{C}_2\subseteq\mathcal{C}_1$ and $l<\omega$ such that $S_\seq{\mathbf{p}, A}\cap C$ has three points which separates $C$ into three arcs each of which has central angle $\geq\frac{2\pi}{l}$ for every $\seq{\mathbf{p}, A}\in\mathcal{C}_2$. Clearly
$l\geq 3$ holds. Now fix a partition $\mathcal{P}$ of $C$ into $2l$ halfopen arcs of the same central angle $\frac{\pi}{l}$.
Then for each $\seq{\mathbf{p}, A}\in\mathcal{C}_2$, $S_\seq{\mathbf{p}, A}$ has points in at least three arcs in $\mathcal{P}$ which are pairwise non-adjacent to each other.
Now pick an uncountable $\mathcal{C}_3\subseteq\mathcal{C}_2$ and three distinct arcs $\gamma_i$ ($0\leq i\leq 2$) in $\mathcal{P}$ which are pairwise non-adjacent to each other, such that $S_\seq{\mathbf{p}, A}$ has points in every $\gamma_i$ ($0\leq i\leq 2$) for every $\seq{\mathbf{p}, A}\in\mathcal{C}_3$. Now pick two distinct configurations $\seq{\mathbf{p}, A}$, $\seq{\mathbf{q}, B}\in\mathcal{C}_3$ and we will show that the stars of them intersect. Pick $\mathbf{r}_i\in S_\seq{\mathbf{p}, A}\cap\gamma_i$ for each $i$ ($0\leq i\leq 2$). Let $\beta_{01}$ be the arc of $C$ with endpoints $\mathbf{r}_0$ and $\mathbf{r}_1$, not containing $\mathbf{r}_2$. Let $U_{01}$ be the open subset of $D$ which has the union of $\beta_{01}$, $[\mathbf{p}, \mathbf{r}_0]$ and $[\mathbf{p}, \mathbf{r}_1]$ as its boundary. Similarly define $\beta_{02}$, $\beta_{12}$ and $U_{02}$, $U_{12}$. Remember that $\mathbf{q}\in D$. If $\mathbf{q}$ is on some $[\mathbf{p}, \mathbf{r}_i]$ we are already done. So we may assume, without loss of generality, that $\mathbf{q}\in U_{01}$. Note that $B$ has a point $\mathbf{s}\in\gamma_2$, but since $\gamma_2$ lies in the exterior of $U_{01}$, $[\mathbf{q}, \mathbf{s}]$ must intersect with the boundary of $U_{01}$. But since $\mathbf{q}\in D$ and $\mathbf{s}\in C$, $[\mathbf{q}, \mathbf{s}]$ cannot intersect with $C$ at points other than $\mathbf{s}$. Therefore $[\mathbf{q}, \mathbf{s}]$ must intersect with either $[\mathbf{p}, \mathbf{r}_0]$ or $[\mathbf{p}, \mathbf{r}_1]$. This shows that $S_\seq{\mathbf{p}, A}$ and $S_\seq{\mathbf{q}, B}$ intersect.\qed

\smallskip
As an application of Theorem \ref{thm:moore}, we show graphs realizable in a $2$-dimensional manifold (in a somewhat strong sense) have a strong structural constraint.
\begin{dfn}
A (simple) {\it graph\/} is a pair $\mathcal{G}=\seq{G, E}$ of a set $G$ and $E\subseteq[G]^2$. For an $\R$-vector space $V$ and $W\subseteq V$, we say a graph $\seq{G, E}$ is {\it linearly realizable\/} in $W$ if there exists an injective map $f:G\to W$ such that
\begin{enumerate}[(1)]
\item For every $e=\{v, w\}\in E$, $(f(v), f(w))\subseteq W\setminus f''G$ (we will write the open line segment $(f(v), f(w))$ as $f(e)$).
\item For every two distinct $e$, $e'\in E$, $f(e)\cap f(e')=\emptyset$.
\end{enumerate}
We say $f$ is a {\it linear realization\/} of $\mathcal{G}$ in $W$. 
\end{dfn}

\begin{cor}\label{cor:deg3} Assume $\mathrm{CC}_\R$. Suppose $D$ is a nonempty open convex subset of $\R^3$. If a graph $\mathcal{G}$ is linearly realizable in $\partial D$, then $\mathcal{G}$ has at most countably many vertices of degree $\geq 3$.
\end{cor}

\proof Suppose $D$ is a nonempty open convex subset of $\R^3$ and $\mathcal{G}=\seq{G, E}$ is a graph linearly realizable in $\partial D$. By identifying the vertices in $\mathcal{G}$ with those of its realization, we may assume $G\subseteq\partial D$. For each $\mathbf{p}\in G$, let $N_\mathbf{p}=\{\mathbf{q}\in G\mid\{\mathbf{p}, \mathbf{q}\}\in E\}$. Now let $G'$ denote the set of vertices of degree $\geq 3$ in the graph $\mathcal{G}$. Suppose $G'$ is uncountable. By $\mathrm{CC}_\R$, there exists $\mathbf{p}_0\in\R^3$ such that $G'\cap U$ is uncountable for every neighborhood $U$ of $\mathbf{p}_0$ in $\R^3$. Since $G'\subseteq\partial D$ and $\partial D$ is closed in $\R^3$, $\mathbf{p}_0\in\partial D$.

\smallskip\noindent
(Claim) There exists an open neighborhood $U_0$ of $\mathbf{p}_0$ in $\partial D$ and a plane $P$ in $\R^3$ such that the orthogonal projection $\pi: \R^3\to P$ is injective on $U_0$.

\smallskip\noindent
(Proof of Claim) Pick $\mathbf{c}\in D$, and let $P$ be the plane containing $\mathbf{c}$ and orthogonal to the line containing $\mathbf{c}$ and $\mathbf{p}_0$. Let $H$ be the open halfspace determined by $P$ containing $\mathbf{p}_0$ and pick an open ball $B$ with center $\mathbf{c}$ such that $B\subseteq D$. Note that $B\cap P$ is an open disc with center $\mathbf{c}$. Let $\pi:\R^3\to P$ be the orthogonal projection, and let $U_0=\pi^{-1}(B\cap P)\cap H\cap\partial D$. Since $H$ is open and $B\cap P$ is open relative to $P$, $U_0$ is an open neighborhood of $\mathbf{p}_0$ in $\partial D$. Note that for each $\mathbf{q}\in B\cap P$, $\pi^{-1}(\mathbf{q})\cap H$ is a halfline with the endpoint $\mathbf{q}\in D$, and thus it can have at most one point in $\partial D$ by Proposition \ref{prop:face}. Therefore $\pi$ is injective on $U_0$.\qed(Claim)

Fix $U_0$, $P$ and $\pi$ as in Claim. By the choice of $\mathbf{p}_0$, $G'\cap U_0$ is uncountable. For each $\mathbf{p}\in G'\cap U_0$ and $\mathbf{q}\in N_\mathbf{p}$, let $n_{\mathbf{p}, \mathbf{q}}$ be the least $n<\omega$ such that
$$
[\mathbf{p}, \frac{(n+2)\mathbf{p}+\mathbf{q}}{n+3}]\subseteq U_0
$$
and let
$$
A_\mathbf{p}=\{\frac{(n_{\mathbf{p}, \mathbf{q}}+2)\mathbf{p}+\mathbf{q}}{n_{\mathbf{p}, \mathbf{q}}+3}\mid\mathbf{q}\in N_\mathbf{p}\}.
$$
Then for every $\mathbf{p}\in G'\cap U_0$, $\seq{\mathbf{p}, A_\mathbf{p}}$ is a star configuration in $U_0$. Moreover since $\mathcal{G}$ is linearly realized, it is clear that $\{S_\seq{\mathbf{p}, A_\mathbf{p}}\mid\mathbf{p}\in G'\cap U_0\}$ is pairwise disjoint. Now since $\pi$ is affine and injective on $U_0$, for each $\mathbf{p}\in G'\cap U_0$, $\seq{\pi(\mathbf{p}), \pi''A_\mathbf{p}}$ is a star configuration in $B\cap P\subseteq P$, and $\{S_\seq{\pi(\mathbf{p}), \pi''A_\mathbf{p}}\mid\mathbf{p}\in G'\cap U_0\}$ is pairwise disjoint as well. This contradicts Theorem \ref{thm:moore}.\qed

It is likely that one can in fact prove a `topological version' of Corollary \ref{cor:deg3}, that is, one can prove that every graph which is `realizable' (in a way edges are realized as simple curves rather than line segments) in a $2$-dimensional manifold has at most countably many vertices of degree $\geq 3$, only using a weak fragment of $\mathrm{AC}$. But we do not work out this version here, because the proof could be much longer, involving us with treatments of some topological theorems like the Jordan Curve Theorem.

%% file: convex_uniftomcv3.tex
In this section we prove the following.
\begin{thm}\label{lma:uniftothree}
$\mathrm{Unif}_\R$ implies $\mathrm{MCV}(3)$. 
\end{thm}
\proof  We proceed in the same way as in the proof of Theorem \ref{thm:cctotwo}.
Suppose $X\subseteq\R^3$ and we will find a maximal convex subset of $X$. We may assume that $X$ contains a convex subset of dimension $3$, because otherwise our task is reduced to the case of $\R^2$. By the same argument as in the proof of Theorem \ref{thm:cctotwo}, we can find a nonempty open maximal convex subset $D$ of $X$, and again let us find a maximal convex subset of $X$ by extending $D$. For this purpose we may assume that $D\subseteq X\subseteq C$, where $C=\operatorname{cl}D$.

Let
\begin{eqnarray*}
\mathcal{G}&=&\text{the set of $2$-faces of $C$,}\\
\mathcal{E}_2&=&\text{the set of $1$-faces of $C$ contained in some $2$-face of $C$,}\\
\mathcal{E}&=&\text{the set of other $1$-faces of $C$,}\\
\mathcal{V}&=&\text{the set of $0$-faces of $C$.} 
\end{eqnarray*}

Note that $|\mathcal{G}|\leq\omega$ by Proposition \ref{prop:basic}(\ref{item:minusone}). Let $\seq{F_i\mid i<l}$ ($l\leq\omega$) be an enumeration of $\mathcal{G}$. Since each $2$-face of $C$ contains at most countably many $1$-faces of $C$ by Proposition \ref{prop:basic}(\ref{item:faceofface}) and (\ref{item:minusone}), we also have $|\mathcal{E}_2|\leq\omega$ (Here we use $\mathrm{CC}_\R$, together with the fact that the set of countable families of $1$-dimensional convex sets in $\R^3$ is $\R$-semiencodable). Let $\seq{E_i\mid i<m}$ ($m\leq\omega$) be an enumeration of $\mathcal{E}_2$.

Now let
\begin{eqnarray*}
\mathcal{F}_0&=&\{C\},\\
\mathcal{F}_1&=&\mathcal{F}_0\cup\mathcal{G},\\
\mathcal{F}_2&=&\mathcal{F}_1\cup\mathcal{E},\\
\mathcal{F}_{i+3}&=&\mathcal{F}_{i+2}\cup\{E_i\}\quad(\text{for $i<m$}),\\
\mathcal{F}_\omega&=&\bigcup_{i<\omega}\mathcal{F}_i\quad(\text{if $m=\omega$}),\\
\mathcal{F}_{2+m+1}&=&\mathcal{F}_{2+m}\cup\mathcal{V}.
\end{eqnarray*}

Then $\seq{\mathcal{F}_i\mid i\leq 2+m+1}$ is a fine face filtration of $C$. Let $\seq{\mathcal{F}^\circ_i\mid i\leq 2+m+1}$
be the fine convex filtration derived as in Lemma \ref{lma:face}. We will find $\seq{W_i\mid i\leq 2+m+1}$ satisfying conditions in Lemma \ref{lma:general}(\ref{item:genmcfc}). Then $W_{2+m+1}$ will be a maximal convex subset of $X$.

\smallskip\noindent
(Choice of $W_0$) By the same argument as in the proof of Theorem \ref{thm:cctotwo} we may set $W_0=D$.

\smallskip\noindent
(Choice of $W_1$) By fineness $W_0\cup\{\mathbf{p}\}=\mathcal{F}^\circ_0\cup\{\mathbf{p}\}$ is convex for each $\mathbf{p}\in\mathcal{F}^\circ_1\setminus\mathcal{F}^\circ_0$. Note that $\mathcal{F}^\circ_1\setminus\mathcal{F}^\circ_0=\coprod_{F\in\mathcal{G}}\mathrm{rint} F$, and for any $\mathbf{p}\in\mathrm{rint} F$ and $\mathbf{q}\in\mathrm{rint}F'$ for distinct $F$, $F'\in\mathcal{G}$, $(\mathbf{p}, \mathbf{q})\subseteq D=W_0$ holds by Proposition \ref{prop:face}(\ref{item:pqrint}) (since $F\lor F'=C$ in the lattice $(\mathcal{F}_C, \subseteq)$). Therefore to find $W_1$ satisfying conditions in Lemma \ref{lma:general}(\ref{item:genmcfc}), we may choose a maximal convex subset $C_F$ of $X\cap\mathrm{rint} F$ for each $F\in\mathcal{G}$ and let
$$
W_1=W_0\cup\bigcup\{C_F\mid F\in\mathcal{G}\}.
$$
By Corollary \ref{cor:ccborel} $\mathrm{CC}_\R$ implies that for each $F\in\mathcal{G}$ there exists a $G_\delta$ subset $S$ of $\mathrm{aff}F$ such that $X\cap S$ is a maximal convex subset of $X\cap F$. Since a $G_\delta$ subset of a plane in $\R^3$ is also $G_\delta$ in $\R^3$, and the set of $G_\delta$ subsets of $\R^3$ is $\R$-semiencodable, we may apply $\mathrm{CC}_\R$ once more to obtain $\seq{S_F\mid F\in\mathcal{G}}$ such that each $S_F$ is a $G_\delta$ subset of $\mathrm{aff}F$ and that $X\cap S_F$ is a maximal convex subset of $X\cap F$. So we may let $C_F=X\cap S_F$ for each $F\in\mathcal{G}$.

\smallskip\noindent
(Choice of $W_2$) For any $\mathbf{p}\in E$ for any $E\in\mathcal{E}$, since $C$ is the only face properly containing $E$, $W_1\cup\{\mathbf{p}\}$ is convex. Moreover for every two distinct $E$, $E'\in\mathcal{E}$, $(\mathbf{p}, \mathbf{q})\subseteq D\subseteq W_1$ holds for any $\mathbf{p}\in E$ and $\mathbf{q}\in E'$. Therefore to find $W_2$ satisfying conditions in Lemma \ref{lma:general}(\ref{item:genmcfc}), we may choose a maximal convex subset $C_E$ of $X\cap\mathrm{rint} E$ for each $E\in\mathcal{E}$ and let
$$
W_2=W_1\cup\bigcup\{C_E\mid E\in\mathcal{E}\}.
$$
Since the set of convex sets of dimension $\leq 1$ is $\R$-encodable, the choice of the family $\{C_E\mid E\in\mathcal{E}\}$ can be done using $\mathrm{Unif}_\R$.

\smallskip\noindent
(Choice of $\seq{W_i\mid 3\leq i\leq 2+m}$) Suppose $W_{i+2}$ was chosen for $i<m$. Then first let
$$
Y_i=\{\mathbf{p}\in X\cap E_i\mid\text{$W_{i+2}\cup\{\mathbf{p}\}$ is convex}\}
$$
and if $Y_i$ is nonempty, choose a connected component $C_i$ of $Y_i$. Then we may let
$$
W_{i+3}=
\begin{cases} 
W_{i+2}\cup C_i&\text{if $Y_i\not=\emptyset$,}\\
W_{i+2}&\text{if $Y_i=\emptyset$.}
\end{cases}
$$
In case $m=\omega$ we set $W_\omega=\bigcup_{i<\omega}W_i$.

If $m<\omega$, we can execute the above process without any fragment of $\mathrm{AC}$. If $m=\omega$, since for each $i<\omega$ we have to choose a convex set $C_i$ of dimension $\leq 1$, from a class depending on the preceding choices. So this can be done using $\mathrm{DC}_\R$, which is a consequence of $\mathrm{Unif}_\R$.

\smallskip\noindent
(Choice of $W_{2+m+1}$) Let
$$
P_0=\{\mathbf{p}\in X\cap\bigcup\mathcal{V}\mid\text{$W_{2+m}\cup\{\mathbf{p}\}$ is convex}\}.
$$
According to Lemma \ref{lma:general}(\ref{item:genmcfc}), we may choose $P\subseteq P_0$ which is maximal with respect to the property that $(\mathbf{p}, \mathbf{q})\subseteq W_{2+m}$ for every two distinct $\mathbf{p}$, $\mathbf{q}\in P$ and let $W_{2+m+1}=W_{2+m}\cup P$.

\smallskip\noindent
(Claim 1) For every two distinct $\mathbf{p}$ and $\mathbf{q}$ in $P_0$, either $(\mathbf{p}, \mathbf{q})\subseteq W_{2+m}$ or $(\mathbf{p}, \mathbf{q})\cap W_{2+m}=\emptyset$ holds.

\smallskip\noindent
(Proof of Claim 1) Suppose there exists $\mathbf{r}\in(\mathbf{p}, \mathbf{q})\cap W_{2+m}$. Then since both $W_{2+m}\cup\{\mathbf{p}\}$ and $W_{2+m}\cup\{\mathbf{q}\}$ are convex, $(\mathbf{p}, \mathbf{r})$ and $(\mathbf{q}, \mathbf{r})$ are also contained in $W_{2+m}$, and thus $(\mathbf{p}, \mathbf{q})\subseteq W_{2+m}$ holds.\qed(Claim 1)

Let $\tilde{\mathcal{E}}=\{\{\mathbf{p}, \mathbf{q}\}\in[P_0]^2\mid (\mathbf{p}, \mathbf{q})\cap W_{2+m}=\emptyset\}$. By Claim 1, the task finding $P$ above is equivalent to find a maximal independent subset of the graph $\mathcal{H}=(P_0, \tilde{\mathcal{E}})$. To do it, we will observe that a sufficiently large part of $\mathcal{H}$ is linearly realizable in $\partial C$ and apply Corollary \ref{cor:deg3} to that part. To this end, as an exception handling, first we will choose a sequence of subsets $\seq{Q_i\mid i\leq l}$ of $P_0$ as follows. Let $Q_0=\emptyset$ first. Suppose $Q_i$ ($i<l$) is given. If there exists $\mathbf{p}\in P_0\cap F_i$ such that
$Q_i\cup\{\mathbf{p}\}$ is independent in $\mathcal{H}$, then pick one such $\mathbf{p}$ and set $Q_{i+1}=Q_i\cup\{\mathbf{p}\}$.
Otherwise, just let $Q_{i+1}=Q_i$. In case $l=\omega$, let $Q_\omega=\bigcup_{i<\omega}Q_i$.

Note that this process to choose $\seq{Q_i\mid i\leq l}$ can be done using $\mathrm{DC}_\R$ (in case $l=\omega$; otherwise we don't need it). Clearly $Q_l$ is independent in $\mathcal{H}$.

Now let
$$
P_1=\{\mathbf{p}\in P_0\setminus Q_l\mid\text{$Q_l\cup\{\mathbf{p}\}$ is independent in $\mathcal{H}$}\}.
$$
Then we may find a maximal independent subset $P'$ of $\mathcal{H}\upharpoonright P_1=(P_1, \tilde{\mathcal{E}}\cap[P_1]^2)$, and let $P=Q_l\cup P'$. Now we will show that $\mathcal{H}\upharpoonright P_1$ is linearly realizable in $\partial C$.

\smallskip\noindent
(Claim 2) For every $F\in\mathcal{G}$, if $P_1\cap F\not=\emptyset$ then $Q_l\cap F\not=\emptyset$ holds.

\smallskip\noindent
(Proof of Claim 2) Suppose $F\in\mathcal{G}$ and $P_1\cap F\not=\emptyset$. Let $i<l$ be such that $F=F_i$ and let $\mathbf{p}\in P_1\cap F$. Then since $\mathbf{p}\in P_0$ and $Q_i\cup\{\mathbf{p}\}$ is independent in $\mathcal{H}$, $Q_{i+1}\subseteq Q_l$ must contain a point in $F_i=F$.\qed(Claim 2)

\smallskip\noindent
(Claim 3) For every $\{\mathbf{p}, \mathbf{q}\}\in\tilde{\mathcal{E}}\cap[P_1]^2$, $1\leq\mathrm{dim}(\{\mathbf{p}\}\lor\{\mathbf{q}\})\leq 2$ holds. In particular $(\mathbf{p}, \mathbf{q})\subseteq\operatorname{rbd} C=\partial C$ and does not intersect with $P_1$.

\smallskip\noindent
(Proof of Claim 3) Suppose $\{\mathbf{p}, \mathbf{q}\}\in\tilde{\mathcal{E}}\cap[P_1]^2$. $\mathrm{dim}(\{\mathbf{p}\}\lor\{\mathbf{q}\})\geq 1$ is clear. Since $(\mathbf{p}, \mathbf{q})\cap W_{2+m}=\emptyset$ we have $(\mathbf{p}, \mathbf{q})\nsubseteq D=\operatorname{rint} C$ and thus it holds that $\mathrm{dim}(\{\mathbf{p}\}\lor\{\mathbf{q}\})\leq 2$. $(\mathbf{p}, \mathbf{q})\subseteq\partial C$ follows from Proposition \ref{prop:basic}(\ref{item:properface}), and $(\mathbf{p}, \mathbf{q})\cap P_1=\emptyset$ since $P_1$ consists only of $0$-faces of $C$.\qed(Claim 3)

\smallskip\noindent
(Claim 4) For every two distinct $\{\mathbf{p}, \mathbf{p}'\}$ and $\{\mathbf{q}, \mathbf{q}'\}\in\tilde{\mathcal{E}}\cap[P_1]^2$, $(\mathbf{p}, \mathbf{p}')\cap(\mathbf{q}, \mathbf{q}')=\emptyset$ holds.

\smallskip\noindent
(Proof of Claim 4) Suppose $\{\mathbf{p}, \mathbf{p}'\}$ and $\{\mathbf{q}, \mathbf{q}'\}$ are distinct members of $\tilde{\mathcal{E}}\cap[P_1]^2$ and $(\mathbf{p}, \mathbf{p}')\cap(\mathbf{q}, \mathbf{q}')\not=\emptyset$. Then by Lemma \ref{lma:face}(\ref{item:pqrint}), $\{\mathbf{p}\}\lor\{\mathbf{p}'\}=\{\mathbf{q}\}\lor\{\mathbf{q}'\}$ holds in the lattice $(\mathcal{F}_C, \subseteq)$. Denote this face as $F$. By Claim 3, it never happens that $(\mathbf{p}, \mathbf{p}')$ and $(\mathbf{q}, \mathbf{q}')$ are on the same line. Therefore $(\mathbf{p}, \mathbf{p}')$ and $(\mathbf{q}, \mathbf{q}')$ intersects at a single point and again by Claim 3, $\operatorname{dim} F=2$ holds. In the plane $\operatorname{aff} F$, $\mathbf{q}$ and $\mathbf{q}'$ lie in opposite sides of the line $\operatorname{aff}((\mathbf{p}, \mathbf{p}'))$. Note that no (relatively) open line segment in $C$ can contain $0$-faces of $C$, and therfore $\operatorname{aff}((\mathbf{q}, \mathbf{q}'))\cap C=[\mathbf{q}, \mathbf{q}']$ holds.

Since $\mathbf{p}\in F\cap P_1\not=\emptyset$, by Claim 2 there exists $\mathbf{r}\in Q_l\cap F$. By definition $\mathbf{r}\notin P_1$, and since $(\mathbf{q}, \mathbf{q}')$ contains neither point of $W_{2+m}$ nor $0$-faces of $C$, $\mathbf{r}\notin(\mathbf{q}, \mathbf{q}')$ as well.Therefore $\mathbf{r}\notin[\mathbf{q}, \mathbf{q}']$ and thus is not on the line $\operatorname{aff}((\mathbf{q}, \mathbf{q}'))$. Then either $(\mathbf{p}, \mathbf{r})$ or $(\mathbf{p}', \mathbf{r})$ must intersect with $\operatorname{aff}((\mathbf{q}, \mathbf{q}'))$ and thus with $[\mathbf{q}, \mathbf{q}']$. But on the one hand $(\mathbf{p}, \mathbf{r})$, $(\mathbf{p}', \mathbf{r})\subseteq W_{2+m}$ by definition of $P_1$, on the other hand $[\mathbf{q}, \mathbf{q}']$ does not intersect with $W_{2+m}$. This is a contradiction.\qed(Claim 4)

By Claim 3 and 4, the graph $\mathcal{H}\upharpoonright P_1$ is linearly realized in $\partial C$. Therefore by Corollary \ref{cor:deg3}, $\mathcal{H}\upharpoonright P_1$ contains at most countably many vertices of degree $\geq 3$.

Now let $R$ be the set of vertices of degree $\geq 3$ in $\mathcal{H}\upharpoonright P_1$, and $\seq{\mathbf{p}_i\mid i<n}$ ($n\leq\omega$) an enumeration of $R$. Define $\seq{R_i\mid i\leq n}$ as follows: Let $R_0=\emptyset$. Suppose $R_i$ ($i<n$) was defined. Let $R_{i+1}=R_i\cup\{\mathbf{p}_i\}$ if $R_i\cup\{\mathbf{p}_i\}$ is independent in $\mathcal{H}\upharpoonright P_1$. Otherwise let $R_{i+1}=R_i$. In case $n=\omega$ let $R_\omega=\bigcup_{i<\omega}R_i$. Now let
$$
P_2=\{\mathbf{p}\in P_1\setminus R_n\mid \text{$R_n\cup\{\mathbf{p}\}$ is independent in $\mathcal{H}\restrict P_1$}\}.
$$
Then to find a maximal independent subset $P'$ of $\mathcal{H}\upharpoonright P_1$, we may find a maximal independent subset $P''$ of $\mathcal{H}\upharpoonright P_2$ and let $P'=R_n\cup P''$.

Let $\mathcal{K}$ denote the set of connected components of the graph $\mathcal{H}\upharpoonright P_2$. To find $P''$ as above, it is enough to find a maximal independent subset $P_K$ of $\mathcal{H}\upharpoonright K$ for each $K\in\mathcal{K}$, and let $P''=\bigcup_{K\in\mathcal{K}}P_K$. We will see that we can do this without using any fragment of $\mathrm{AC}$. Note that in the graph $\mathcal{H}\upharpoonright P_2$, all vertices are of degree $\leq 2$, and thus each $K\in\mathcal{K}$ satisfies exactly one of the following:
\begin{enumerate}[(a)]
\item $K$ is finite.
\item $\mathcal{H}\upharpoonright K$ is isomorphic to $(\omega, \mathcal{N}_\omega)$, where $\mathcal{N}_\omega=\{\{n, n+1\}\mid n\in\omega\}$.
\item $\mathcal{H}\upharpoonright K$ is isomorphic to $(\mathbb{Z}, \mathcal{N}_\mathbb{Z})$, where $\mathcal{N}_\mathbb{Z}=\{\{n, n+1\}\mid n\in\mathbb{Z}\}$.
\end{enumerate}

\smallskip\noindent
(Case 1) $K$ is finite.

Let $\{\mathbf{q^K}_i\mid i<|K|\}$ be the enumeration of $K$ ordered by $<_3$. Then define $\seq{S^K_i\mid i\leq|K|}$ as follows. Let $S^K_0=\emptyset$. Suppose $S^K_i$ ($i<|K|$) was defined. Then let $S^K_{i+1}=S^K_i\cup\{q^K_i\}$ if $S^K_i\cup\{q^K_i\}$ is independent in $\mathcal{H}\upharpoonright K$. Otherwise let $S^K_{i+1}=S^K_i$. Then let $P_K=S^K_{|K|}$, which is a maximal independent subset of $\mathcal{H}\upharpoonright K$.

\smallskip\noindent
(Case 2) $\mathcal{H}\upharpoonright K$ is isomorphic to $(\omega, \mathcal{N}_\omega)$.

Let $f:\omega\to K$ be the unique isomorphism between $(\omega, \mathcal{N}_\omega)$ and $\mathcal{H}\upharpoonright K$. Then let $P_K=\{f(2n)\mid n<\omega\}$. Clearly it is a maximal independent subset of $\mathcal{H}\upharpoonright K$.

\smallskip\noindent
(Case 3) $\mathcal{H}\upharpoonright K$ is isomorphic to $(\mathbb{Z}, \mathcal{N}_\mathbb{Z})$.

Fix an enumeration $\seq{r_j\mid j<\omega}$ of the rational numbers.
Let $i_0$ be the least $i<3$ such that $|\{\pi_i(\mathbf{p})\mid\mathbf{p}\in K\}|\geq 2$, where $\pi_i(\mathbf{p})$ denotes the $i$-th coordinate of $\mathbf{p}$. Let $j_0$ be the least $j<\omega$ such that
$$
A^K_j=\{\mathbf{p}\in K\mid \pi_{i_0}(\mathbf{p})\leq r_j\}\ \text{and}\ B^K_j=\{\mathbf{p}\in K\mid \pi_{i_0}(\mathbf{p})>r_j\}
$$
are both nonempty. Let $\mathcal{K}_{A^K_{j_0}}$ and $\mathcal{K}_{B^K_{j_0}}$ be the set of connected components of the graphs $\mathcal{H}\upharpoonright A^K_{j_0}$ and $\mathcal{H}\upharpoonright B^K_{j_0}$ respectively. Note that each $M\in\mathcal{K}_{A^K_{j_0}}$ satisfies (a) or (b) above, and thus we can define a maximal independent subset $I_M$ of $\mathcal{H}\upharpoonright M$ using the above procedure. Let $P_{A^K_{j_0}}=\bigcup_{M\in\mathcal{K}_{A^K_{j_0}}}I_M$.

Now let $N$ be any connected component of $\mathcal{H}\upharpoonright B^K_{j_0}$. Let
$$
N'=\{\mathbf{p}\in N\mid\text{$\{\mathbf{p}, \mathbf{q}\}\in\tilde{\mathcal{E}}$ for some $\mathbf{q}\in P_{A^K_{j_0}}$}\}.
$$
Note that $N\setminus N'$ satisfies (a) or (b) unless $N\setminus N'=\emptyset$. Then using the above procedure, define a maximal independent subset $J_N$ of $\mathcal{H}\upharpoonright (N\setminus N')$. If $N\setminus N'=\emptyset$, just let $J_N=\emptyset$. Now let $P_K=P_{A^K_{j_0}}\cup\bigcup_{N\in\mathcal{K}_{B^K_{j_0}}}J_N$. It is easy to see that $P_K$ is a maximal independent subset of $\mathcal{H}\upharpoonright K$.

This completes our proof of Lemma \ref{lma:uniftothree}.\qed

\begin{cor}
$\mathrm{MCV(3)}$ is equivalent to $\mathrm{Unif}_\R$.\qed
\end{cor}

%% file: convex_higher.tex
In this section we discuss $\mathrm{MCV}(V)$ for some $V$'s of infinite dimensions.

First let us observe that $\mathrm{MCV}(2^\omega)$ and $\mathrm{MCV}(2^{2^\omega})$ can be respectively understood as statements for more familiar spaces.

\begin{prop}\label{prop:familiar}
\begin{enumerate}[(1)]
\item\label{item:fam2omega} $\mathrm{MCV}(2^\omega)$ is equivalent to $\mathrm{MCV}(\R^\omega)$.
\item\label{item:fam22omega} $\mathrm{MCV}(2^{2^\omega})$ is equivalent to $\mathrm{MCV}(\R^\R)$.
\end{enumerate}
\end{prop}

\proof (\ref{item:fam2omega}) Let $V$ be an $\R$-vector space of dimension $2^\omega$. Since $|\R^\omega|=2^\omega$, there is a surjective $\R$-linear map from $V$ to $\R^\omega$. On the other hand, $\R^\omega$ has a linearly independent subset of size $2^\omega$: Let $\mathcal{I}$ be an independent family of subsets of $\omega$ of size $2^\omega$ (Fichtenholz and Kantorovich \cite{fichtenholz35:_independent} showed that such $\mathcal{I}$ exists without using $\mathrm{AC}$. An alternative proof was given by Hausdorff \cite{hausdorff36:_uberzwei}. See also Geschke \cite{geschke12:_adandif}). For each $x\in\mathcal{I}$ let $\mathbf{p}_x\in\R^\omega$ be the characteristic function of $x\subset\omega$. Then it is easy to see $\{\mathbf{p}_x\mid x\in\mathcal{I}\}$ is linearly independent. This shows that there is an injective $\R$-linear map from $V$ to $\R^\omega$. Thus our conclusion follows from Proposition \ref{prop:vw}.

(\ref{item:fam22omega}) It is proved in a similar way as above. On the one hand we have $|\R^\R|=2^{2^\omega}$, and on the other hand we have a linearly independent subset of $\R^\R$ of size $2^{2^\omega}$, since it is proved that there exists an independent family of subsets of $\R$ of size $2^{2^\omega}$ without using $\mathrm{AC}$ (see \cite{fichtenholz35:_independent}, \cite{hausdorff36:_uberzwei} or \cite{geschke12:_adandif}).\qed

Note that by similar arguments one can show that $\mathrm{MCV}(2^\omega)$ is equivalent to $\mathrm{MCV}(V)$ for many $V$'s which are shown to have dimension $2^\omega$ under $\mathrm{AC}$, like the space $l^\infty$ of bounded sequences of reals, the space $C(\R, \R)$ of continuous functions from $\R$ to $\R$, the separable infinite dimensional Hilbert space $\mathcal{H}$, and so on. Though without $\mathrm{AC}$ these spaces are not necessarily shown to be mutually isomorphic, one can show all of them have size $2^\omega$ and have a linear independent subset of size $2^\omega$ without $\mathrm{AC}$.

Now we compare $\mathrm{MCV}(V)$'s with combinatorial fragments of $\mathrm{AC}$. 

\begin{lma}\normalfont\label{lma:mcsc}
\begin{enumerate}[(1)]
\item\label{item:mcvsc} For any cardinal $\kappa$, $\mathrm{MCV}(\kappa)$ implies $\mathrm{SC}_\kappa$.
\item\label{item:scmcv} For any $\R$-vector space $V$, $\mathrm{SC}_{V}$ implies $\mathrm{MCV}(V)$.
\end{enumerate}
\end{lma}

\proof (\ref{item:mcvsc}) Let $\kappa$ be any cardinal and let $V=\R_\mathcal{B}$ be an $\R$-vector space such that  $|\mathcal{B}|=\kappa$. Assume $\mathrm{MCV}(\kappa)$ that is equivalent to $\mathrm{MCV}(V)$. Let $C\subseteq[\mathcal{B}]^{<\omega}$ be such that $\emptyset\in C$.
First let
$$
\tilde{C}=\{S\in C\mid\mathcal{P}(S)\subseteq C\}.
$$
Note that $\tilde{C}$ is closed under subset, and that for every $P\subseteq\mathcal{B}$, $[P]^{<\omega}\subseteq C$ holds iff $[P]^{<\omega}\subseteq\tilde{C}$ holds. In particular $\emptyset\in\tilde{C}$ holds. Now let
$$
X_{\tilde{C}}=\bigcup\{\operatorname{conv}_V\{\mathbf{b}_r\mid r\in S\}\mid S\in \tilde{C}\}.
$$
Apply $\mathrm{MCV}(V)$ to obtain a maximal convex subset $W$ of $X_{\tilde{C}}$. Now set
$$
P_W=\{r\in\mathcal{B}\mid\exists\mathbf{p}\in W[\mathbf{p}(r)>0]\}.
$$
Let $S=\{s_0, \ldots, s_k\}\in[P_W]^{<\omega}$ ($k<\omega$) be arbitrary. For each $i\leq k$, we may pick $\mathbf{p}_i\in W$ such that $\mathbf{p}_i(s_i)>0$. Then let $\mathbf{p}=\frac{1}{k+1}\sum^k_{i=0}\mathbf{p}_i$.
Since $W$ is convex, $\mathbf{p}\in W\subseteq X_{\tilde{C}}$ holds. Since all functions in $X_{\tilde{C}}$ are nonnegative, $\mathbf{p}(s_i)>0$ holds for every $i\leq k$. Pick $S'\in\tilde{C}$ such that $\mathbf{p}\in\operatorname{conv}_V\{\mathbf{b}_r\mid r\in S'\}$.
Then $S\subseteq\{r\in\mathcal{B}\mid\mathbf{p}(r)>0\}\subseteq S'$ and thus $S\in\tilde{C}$. This shows that $[P_W]^{<\omega}\subseteq\tilde{C}$. Note that
\begin{eqnarray*}
W&\subseteq&\bigcup\{\operatorname{conv}_V\{\mathbf{b}_r\mid r\in S\}\mid S\in[P_W]^{<\omega}\}\\
&=&\operatorname{conv}_V\{\mathbf{b}_r\mid r\in P_W\}\\
&\subseteq&\operatorname{conv}_V\{\mathbf{b}_r\mid r\in P\}\\
&=&\bigcup\{\operatorname{conv}_V\{\mathbf{b}_r\mid r\in S\}\mid S\in[P]^{<\omega}\}\\
&\subseteq&X_{\tilde{C}}
\end{eqnarray*}
holds. Since $\operatorname{conv}_V\{\mathbf{b}_r\mid r\in P\}$ is a convex subset of $X_{\tilde{C}}$ containing $W$, by maximality of $W$ it holds that
$$
W=\operatorname{conv}_V\{\mathbf{b}_r\mid r\in P_W\}=\operatorname{conv}_V\{\mathbf{b}_r\mid r\in P\}
$$
and therefore $P=P_W$ holds. Thus $P_W$ is a maximal subset of $\mathcal{B}$ such that $[P_W]^{<\omega}\subseteq\tilde{C}$. This shows $\mathrm{SC}_\mathcal{B}$ that is equivalent to $\mathrm{SC}_\kappa$.

(\ref{item:scmcv}) Let $V$ be any $\R$-vector space and assume $\mathrm{SC}_V$. Let $X\subseteq V$ be arbitrary. Let
$$
C=\{S\in[X]^{<\omega}\mid\operatorname{conv}_VS\subseteq X\}.
$$
Note that $C$ is a subset of $[X]^{<\omega}$ and $\emptyset\in C$ holds. Apply $\mathrm{SC}_V$ to obtain a maximal $P\subseteq X$ such that $[P]^{<\omega}\subseteq C$. Let $\mathbf{p}$, $\mathbf{q}\in P$ and $\mathbf{r}\in(\mathbf{p}, \mathbf{q})$. Since $\{\mathbf{p}, \mathbf{q}\}\in C$, $\mathbf{r}\in\operatorname{conv}_V\{\mathbf{p}, \mathbf{q}\}\subseteq X$ holds. Moreover, For any $S\in[P]^{<\omega}$ we have
$$
\operatorname{conv}(S\cup\{\mathbf{r}\})\subseteq\operatorname{conv}(S\cup\{\mathbf{p}, \mathbf{q}\})\subseteq X
$$
and thus $S\cup\{\mathbf{r}\}\in C$. Therefore $[P\cup\{\mathbf{r}\}]^{<\omega}\subseteq C$ and by maximality of $P$ we have $\mathbf{r}\in P$. This shows that $P$ is a convex subset of $X$. Now suppose $P'\subseteq X$ is convex and contains $P$. Then for every $S\in[P']^{<\omega}$, $\operatorname{conv}_VS\subseteq P'\subseteq X$ holds and thus $S\in C$. Therefore we have $[P']^{<\omega}\subseteq C$, and thus by maximality of $P$ we have $P'=P$. This shows that $P$ is a maximal convex subset of $X$. Therefore we have $\mathrm{MCV}(V)$.\qed

\begin{thm}
\begin{enumerate}[(1)]
\item\label{item:cor2omega} $\mathrm{MCV}(2^\omega)$ is equivalent to $\mathrm{SC}_\R$.
\item\label{item:cor22omega} $\mathrm{MCV}(2^{2^\omega})$ is equivalent to $\mathrm{SC}_{\mathcal{P}(\R)}$.
\end{enumerate}
\end{thm}

\proof (\ref{item:cor2omega}) is clear by Lemma \ref{lma:mcsc}, because the $\R$-vector space of dimension $2^\omega$ is of size $2^\omega$. As for (\ref{item:cor22omega}), by Proposition \ref{prop:familiar}(\ref{item:fam22omega}) $\mathrm{MCV}(2^{2^\omega})$ is equivalent to $\mathrm{MCV}(\R^\R)$ and thus by $|\R^\R|=2^{2^\omega}$ and Lemma \ref{lma:mcsc} the conclusion follows.\qed

We also have some consequences on the existence of bases for some $\R$-vector spaces.

\begin{prop}\label{prop:scbasis}
For each $\R$-vector space $V$, $\mathrm{SC}_V$ implies that every linearly independent subset of $V$ can be extended to a basis of $V$.
\end{prop}

\proof Assume $\mathrm{SC}_V$ and let $B_0$ be any linearly independent subset of $V$. Let
$$
C=\{S\in[V\setminus B_0]^{<\omega}\mid\text{$S\cup B_0$ is linearly independent}\}.
$$
Then $C$ is a subset of $[V]^{<\omega}$ and $\emptyset\in C$ holds. Now apply $\mathrm{SC}_V$ to obtain a maximal $P\subseteq V$ such that $[P]^{<\omega}\subseteq C$. Then it is easy to see that $P\cup B_0$ is a basis of $V$.\qed

By Lemma \ref{lma:mcsc}(\ref{item:mcvsc}) and Proposition \ref{prop:scbasis} we have
\begin{cor}\label{cor:mcbasis}
For each $\R$-vector space $V$, if $\mathrm{MCV}(|V|)$ holds then every linearly independent subset of $V$ can be extended to a basis.\qed
\end{cor}

By the above corollary and the remark below the proof of Proposition \ref{prop:familiar} we have the following.

\begin{cor}
Suppose $\mathrm{MCV}(V)$ holds for either $V=\R_\R$, $\R^\omega$, $l^\infty$, $C(\R, \R)$ or $\mathcal{H}$. Then these spaces are all isomorphic (as $\R$-vector spaces).\qed
\end{cor}

Now let us mention $\mathrm{MCV}(\omega_1)$. By Proposition \ref{prop:vw}(\ref{item:surj}), $\mathrm{MCV}(2^\omega)$ implies $\mathrm{MCV}(\omega_1)$. The following gives a lower bound for the strength of $\mathrm{MCV}(\omega_1)$.



\begin{thm}\normalfont
$\mathrm{MCV}(\omega_1)$ implies that $\omega_1\leq 2^\omega$.
\end{thm}
\proof Assume $\mathrm{MCV}(\omega_1)$. Fix a surjection $f:(0, 1)\setminus\Q\to\omega_1$. Let
$$
D=\{\mathbf{p}\in \R_{\omega_1}\mid \mathbf{p}\geq 0\land|\{\alpha<\omega_1\mid\mathbf{p}(\alpha)>0\}|\geq 2\}
$$
and for each $\alpha<\omega_1$ let $B_\alpha=\{r\mathbf{b}_\alpha\mid r\in f^{-1}(\alpha)\}$. Let
$$
X=D\cup\bigcup_{\alpha<\omega_1}B_\alpha.
$$
Then it is easy to see that
\begin{enumerate}[(1)]
\item For every $\mathbf{p}\in X$ and $\mathbf{q}\in D$, $(\mathbf{p}, \mathbf{q})\subseteq D$ holds.
\item For every two distinct $\mathbf{p}$, $\mathbf{q}\in B_\alpha$ for $\alpha\in\omega_1$, $(\mathbf{p}, \mathbf{q})\nsubseteq X$ holds.
\item For every $\mathbf{p}\in B_\alpha$, $\mathbf{q}\in B_\beta$ for distinct $\alpha$, $\beta\in\omega_1$, $(\mathbf{p}, \mathbf{q})\subseteq D$ holds.
\end{enumerate}
By $\mathrm{MCV}(\omega_1)$, $X$ has a maximal convex subset, and by Lemma \ref{lma:generalchoice} there exists a right inverse $g:\omega_1\to(0, 1)\setminus\Q$ of $f$. Since $g$ is injective we have $\omega_1\leq 2^\omega$ holds.\qed


%% file: convex_summary.tex
The following diagram indicates implications between $\mathrm{MCV}$, $\mathrm{MCV}(V)$ for some $V$'s and other fragments of $\mathrm{AC}$. In the diagram, $(\kappa)$ for a cardinal $\kappa$ denotes $\mathrm{MCV}(\kappa)$. $(\mathrm{WO})$ denotes the statement that $\mathrm{MCV}(\kappa)$ holds for every well-orderable cardinal $\kappa$. For an $\R$-vector space $V$, $\mathrm{BS}_V$ denotes the statement that every linearly independent subset of $V$ can be extended to a basis of $V$. $\mathrm{ULF}_\omega$ and $\mathrm{LNM}$ respectively denote the statement that every filter on $\omega$ can be extended to an ultrafilter and the existence of a Lebesgue non-measurable subset of $\R$.

\begin{figure}[h]
\includegraphics[height=9cm]{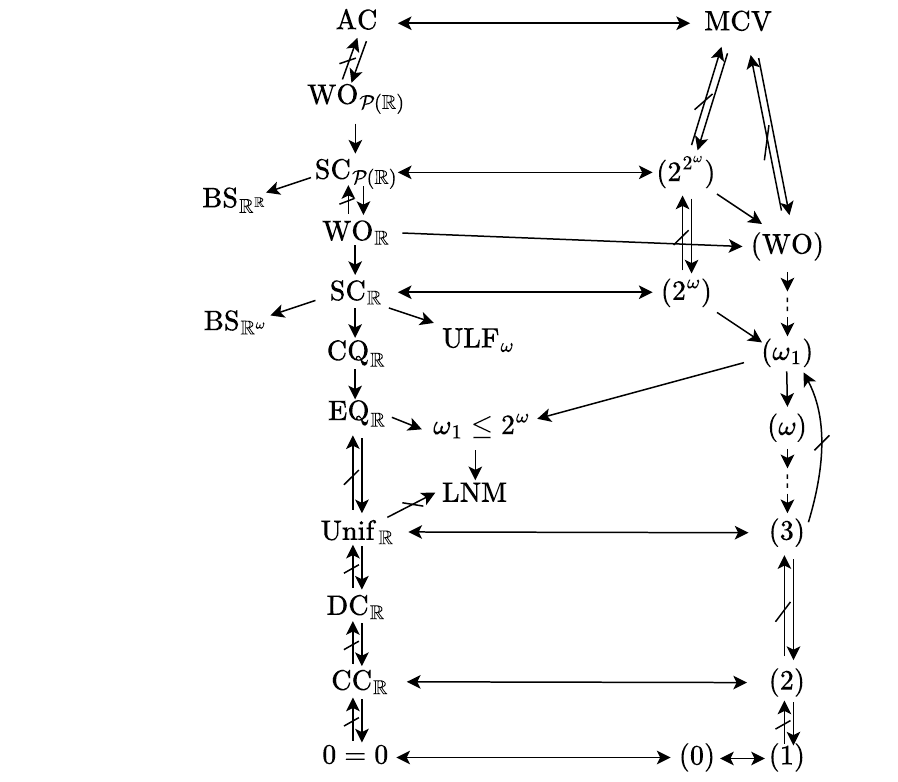}
\end{figure}

Here we list some questions related to the subject of this paper.

\begin{qtn}
How strong is $\mathrm{MCV}(4)$? Is it equivalent to $\mathrm{MCV}(3)$? Or does $\mathrm{MCV}(4)$ (or any $\mathrm{MCV}(n)$ for some $n\leq\omega$) imply the existence of `non-regular' subsets of $\R$ (like Lebesgue non-measurable sets, sets without the property of Baire, or sets without the perfect set property)?
\end{qtn}
\begin{qtn}
Does the statement that $\mathrm{MCV}(\kappa)$ holds for every well-orderable cardinal $\kappa$ imply $\mathrm{WO}_\R$? Is the statement comparable with $\mathrm{MCV}(2^\omega)$?
\end{qtn}
\begin{qtn}
Can any implication between $\mathrm{WO}_\R$, $\mathrm{SC}_\R$, $\mathrm{CQ}_\R$ and $\mathrm{EQ}_\R$ be inverted? How about $\mathrm{WO}_{\mathcal{P}(\R)}$, $\mathrm{SC}_{\mathcal{P}(\R)}$, $\mathrm{CQ}_{\mathcal{P}(\R)}$, $\mathrm{EQ}_{\mathcal{P}(\R)}$ and $\mathrm{Unif}_{\mathcal{P}(\R)}$?
\end{qtn}
\begin{qtn}
Let $\mathrm{MCV}(V, \Gamma)$ denote the restriction of $\mathrm{MCV}(V)$ to the case the given set is in the pointclass $\Gamma$. How strong is it if $V$ is a Polish space and $\Gamma$ is some Borel/Luzin pointclass? What can we say about pointclasses to which maximal convex subsets belong?
\end{qtn}
\begin{qtn}
Does $\mathrm{MCV}(V)$ imply that $V$ has a basis? By Corollary \ref{cor:mcbasis} $\mathrm{MCV}(|V|)$ does, but it is not clear if $\mathrm{MCV}(V)$ does. Is there any implication of the other direction, that is, is $\mathrm{MCV}(V)$ derived from the existence of a basis for some $\R$-vector space? 
\end{qtn}

%% file: convex_acknowledgements.tex
The author would like to thank Asaf Karagila, who read an earlier version of this paper and made many useful comments, calling the author's attention to some papers to reference and to the question how $\mathrm{MCV}$-type axioms relates to the existence of bases for $\R$-vector spaces. We would like to thank Daisuke Ikegami and Toshimichi Usuba for their helpful comments as well.

%% file: locco.bib
@article{andrettanotaro25:_doesdc,
	author = {Andretta, Alessandro and Notaro, Lorenzo},
	journal = {Journal of Symbolic Logic},
	number = {4},
	pages = {1538-1562},
	title = {Does $\mathrm{DC}$ imply $\mathrm{AC}_\omega$, uniformly?},
	volume = {90},
	year = {2025}}

@article{bellfremlin72:_geometric,
	author = {Bell, J. L. and Fremlin, D. H.},
	journal = {Fundamenta Mathematicae},
	number = {2},
	pages = {167-170},
	title = {A geometric form of the axiom of choice},
	volume = {77},
	year = {1972}}

@article{friedman19:_models,
	author = {Friedman, S.-D. and Gitman, V. and Kanovei, V.},
	journal = {J. Math. Log.},
	number = {1},
	pages = {1850013},
	title = {A model of second-order arithmetic satisfying {A}{C} but not {D}{C}},
	volume = {19},
	year = {2019}}

@article{cohen6364:_indep,
	author = {Cohen, Paul},
	journal = {Proc. Natl. Acad. Sci.},
	title = {The independence of the {C}ontinuum {H}ypothesis, {I} \& {I}{I}},
	volume = {50: 1143-1148, 1963; 51: 105-110, 1964}}

@article{jensen67:_consistency,
	author = {Jensen, R. B.},
	journal = {Notices Amer. Math. Soc.},
	pages = {137},
	title = {Consistency results for $\mathrm{ZF}$},
	volume = {14},
	year = {1967}}

@article{solovay:_lebesgue,
	author = {Solovay, Robert M.},
	journal = {Annals of Math., 2nd Ser.},
	number = {1},
	pages = {1-56},
	title = {A model of set-theory in which every set of reals is {L}ebesgue measurable},
	volume = {92},
	year = {1970}}

@book{moore:_ac,
	author = {Moore, Gregory H.},
	publisher = {Springer},
	series = {Studies in the History of Mathematics and Physical Sciences},
	title = {Zermelo's {A}xiom of {C}hoice: {I}ts {O}rigins, {D}evelopment, and {I}nfluence},
	volume = {8},
	year = {1982 (Reprinted by Dover Publications, 2013)}}

@incollection{solovay:_independence,
	author = {Solovay, Robert M.},
	booktitle = {Cabal Seminar 76-77: Proceedings, Caltech-UCLA Logic Seminar 1976-1977},
	editor = {Kechris, Alexander S. and Moschovakis, Yiannis N.},
	publisher = {Springer-Verlag},
	series = {Lecture notes in {M}athematics},
	title = {The independence of {DC} from {AD}},
	volume = {689},
	year = {1978}}

@book{jech:_choice,
	author = {Jech, Thomas J.},
	publisher = {North-Holland Pub. Co., American Elsevier Pub. Co.},
	series = {Study in Logic and the Foundations of Mathematics},
	title = {Axiom of Choice},
	volume = {75},
	year = {1973 (Reprinted by Dover Publications, 2008)}}

@book{herrlich:_choice,
	author = {Herrlich, Horst},
	publisher = {Springer},
	series = {Lecture notes in {M}athematics},
	title = {Axiom of Choice},
	volume = {1876},
	year = {2009}}

@book{howardrubin:_consequence,
	author = {Howard, Paul and Rubin, Jean E.},
	publisher = {American Mathematical Society},
	series = {Mathematical {S}urveys and {M}onographs},
	title = {Consequences of the {A}xiom of {C}hoice},
	volume = {59},
	year = {1998}}

@article{hausdorff36:_uberzwei,
	author = {Hausdorff, F.},
	journal = {Studia Math.},
	pages = {18-19},
	title = {\"{U}ber zwei {S}\"atze von {G}.~{F}ichtenholz and {L}.~{K}antrovich},
	volume = {6},
	year = {1936}}

@article{geschke12:_adandif,
	author = {Geschke, Stefan},
	journal = {RIMS Kokyuroku},
	pages = {1-9},
	title = {Almost disjoint and independent families},
	volume = {1790},
	year = {2012}}

@article{fichtenholz35:_independent,
	author = {Fichtenholz, G. M. and Kantorovich, L. V.},
	journal = {Studia Math.},
	pages = {69-98},
	title = {Sur le op\'{e}rations lin\'{e}ares dans l'espace de fonctions born\'{e}es},
	volume = {5},
	year = {1935}}

@article{moore1928:_triods,
	author = {Moore, R. L.},
	journal = {Proc. Natl. Acad. Sci.},
	number = {1},
	pages = {85-88},
	title = {Concerning triods in the plane and the junction points of plane continua},
	volume = {14},
	year = {1928}}

@book{soltan2020:_convexsets,
	address = {Hackensack},
	author = {Soltan, V.},
	edition = {second},
	publisher = {World Scientific},
	title = {Lectures on {C}onvex {S}ets},
	year = {2020}}
